\documentclass[12pt,halfline,a4paper]{ouparticle}

\usepackage{float}
\usepackage{indentfirst}
\usepackage{amsmath, amsthm, amssymb, amsfonts}
\usepackage{mathtools}
\usepackage{physics}
\usepackage{graphicx}
\usepackage{verbatim}
\usepackage[hidelinks]{hyperref}
\usepackage{color,algorithm,algorithmic}
\usepackage[nottoc]{tocbibind}

\usepackage{caption}
\usepackage{subcaption}
\usepackage{comment}

\newtheorem{theorem}{Theorem}
\newtheorem{lem}{Lemma}

\newcommand{\RNum}[1]{\uppercase\expandafter{\romannumeral #1\relax}}

\def\k{\mathop{\bf K}}
\def\K{\mathop{\mbox{\bf\Large K}}}

\def\leq{\leqslant}
\def\geq{\geqslant}

\def\C{\mathbb C}

\def\N{\mathbb N}

\def\Z{\mathbb Z}

\setbox0=\hbox{$+$}
\newdimen\plusheight
\plusheight=\ht0
\def\+{\mathbin{\lower\plusheight\hbox{$+$}}}

\setbox0=\hbox{$-$}
\newdimen\minusheight
\minusheight=\ht0
\def\-{\mathbin{\lower\plusheight\hbox{$-$}}}

\setbox0=\hbox{$\cdots$}
\newdimen\cdotsheight
\cdotsheight=\plusheight
\def\cds{\lower\cdotsheight\hbox{$\cdots$}}

\setbox0=\hbox{$\cdots$}
\newdimen\cdotsheight
\cdotsheight=\plusheight
\def\dds{\lower\cdotsheight\hbox{$\ddots$}}

\newcommand{\Fhyp}[3]{\,{}_3F_2\!\left(
\begin{matrix}
#1\\
#2
\end{matrix}
; #3
\right)}

\begin{document}

\title{Four explicit continued fractions for values of the Lerch transcendent and the Hurwitz zeta function}

\author{
Andrius Grigutis
\thanks{Institute of Mathematics, Vilnius University,
Naugarduko 24, LT-03225, Vilnius, Lithuania.
E-mails: \texttt{andrius.grigutis@mif.vu.lt},
\texttt{lukas.turcinskas@mif.stud.vu.lt}.}
\quad
Yuri Matiyasevich
\thanks{St.\,Petersburg Department of Steklov Mathematical Institute of Russian Academy of Sciences, Fontanka 27, 191023, St. Petersburg, Russia.
E-mail: \texttt{yumat@pdmi.ras.ru}.}
\quad
Lukas Turčinskas\,\footnotemark[1]
}

\abstract{We prove four explicit continued fraction representations for the Lerch
transcendent $\Phi(z, s, M+1)$, where $\Re M>0$ and
$(z,s)\in\{(-1,1),(-1,2),(1,2),(1,3)\}$. All of the continued fractions have unit partial numerators, while partial denominators depend on the parameter $M$. The proofs combine equivalent
transformations of continued fractions, generalized hypergeometric
functions, three-term recurrences, and asymptotic analysis of minimal
solutions. For $z=1$, the corresponding representations give continued
fractions for the values of the Hurwitz zeta function $\zeta(2, M+1)$ and
$\zeta(3, M+1)$.}

\date{\today}


\maketitle

\section{Short Introduction to Continued Fractions}

Continued fractions are a fundamental tool in rational approximation. Their expansions of irrational real numbers provide rapid rational approximations with an arbitrary precision. The importance of continued fractions extends beyond approximation, including applications to Diophantine equations, quadratic irrationals, and other problems in number theory \cite{Lorentzen}. Continued fractions also provide representations of many important special functions, such as hypergeometric and related functions \cite{BerndtLamphereWilson1985}, \cite{Handbook}. Their expansions are useful both for theoretical investigations and for numerical computation \cite{Lorentzen}, \cite{Handbook}.

Let $(a_1,\,a_2,\,\ldots)$ and $(b_0,\,b_1,\,\ldots)$ denote sequences of numbers. We use the following equivalent notations for the continued fraction:
\begin{align*}
b_0 + \cfrac{a_1}{b_1 + \cfrac{a_2}{b_2 + \ddots}},\qquad 
b_0+\frac{a_{1}}{b_{1}}\+ \frac{a_{2}}{b_{2}}\+
\frac{a_{3}}{b_{3}}\+\cds,\qquad 
b_0 + \K_{n=1}^{\infty}\frac{a_n}{b_n},  \qquad b_0+\K\left(\frac{a_n}{b_n}\right),
\end{align*}
where $a_n$, $b_n$, $n\in\mathbb{N}$ are called the partial numerators and the partial denominators, respectively, and $b_0$ is called the initial term. The counterparts $f_n,\,n\in\mathbb{N}_0$ of partial sums in continued fractions are called the convergents. They can be given by
\[
f_0=b_0, \qquad f_n=\frac{A_n}{B_n}=b_0 + \K_{j=1}^{n}\frac{a_j}{b_j},\qquad n\in\mathbb{N},
\]
where $A_n$ and $B_n$ are defined recursively \cite[Sec.~1.1.3]{Lorentzen} by
\begin{align}\label{fundamental_recurrence}
\begin{aligned}
&A_{-1}=1,\qquad A_0=b_0,\qquad
A_n=b_nA_{n-1}+a_nA_{n-2},\\
&B_{-1}=0,\qquad B_0=1,\qquad
B_n=b_nB_{n-1}+a_nB_{n-2},
\end{aligned}
\qquad n\in\mathbb N.
\end{align}

The convergence of a continued fraction is understood as the convergence of the sequence of its convergents $f_n$, $n\in\mathbb{N}_0$. Whenever the limit $\lim_{n\to\infty}f_n$
exists, the continued fraction is said to converge, and its value is equal to this limit, i.e.,
\[
b_0+\K_{n=1}^{\infty}\frac{a_n}{b_n}
=\lim_{n\to\infty}f_n.
\]

In this work, we prove four identities (Theorems \ref{conj_1}--\ref{conj_4}) relating the explicit continued-fraction representations and specific values of the Lerch transcendent. The Lerch transcendent is defined by
\begin{align}\label{Lerch_def}
\Phi(z,\,s,\,a)=\sum_{n=0}^{\infty}\frac{z^n}{(n+a)^s},
\end{align}
where $z,\,s,\,a\in\mathbb C$ and $a\notin\{0,\,-1,\,-2,\,\ldots\}$. The series \eqref{Lerch_def} converges absolutely for $|z|<1$ and arbitrary $s\in\mathbb{C}$, and for $|z|=1$ whenever
$\Re s >1$; \cite[pp.~28--30]{LaurincikasGarunkstis2002}.

\section{Main results, their origin, and prior work}

Throughout this section, $\Phi(z, s, a)$ and $\zeta(s, a)=\Phi(1, s, a)$ denote the Lerch transcendent and the Hurwitz zeta function, respectively. 

\subsection{Main results}
In the following four theorems, we summarize the main results of the paper.

\begin{theorem}\label{conj_1}
Let $M\in\mathbb C$ satisfy $\Re M >0$, and let
\[
b_0=0,\qquad
b_1=1,\qquad
b_2=2M-1, \qquad
b_n=\frac{2M}{n-1},\qquad n\in \{3,\,4,\,\ldots\}.
\]
Then
\begin{equation}\label{cf1}
2M\Phi(-1,\,1,\,M+1)
=
\K_{n=1}^{\infty}\frac{1}{b_n}.
\end{equation}
\end{theorem}

\begin{theorem}\label{conj_2}
Let $M\in\mathbb C$ satisfy $\Re M >0$, and let
\[
b_0=0,\qquad
b_{2n-1}=M+1,\qquad
b_{2n}=\frac{M}{n^2},\qquad n\in\mathbb{N}.
\]
Then
\begin{equation}\label{cf2}
2M\Phi(-1,\,2,\,M+1)
=
\K_{n=1}^{\infty}\frac{1}{b_n}.
\end{equation}
\end{theorem}

\begin{theorem}\label{conj_3}
Let $M\in\mathbb C$ satisfy $\Re M >0$, and let
\begin{align*}
b_0=2M-1,\qquad
b_n=2M\frac{2n+1}{n(n+1)},\qquad n\in\mathbb N.
\end{align*}
Then
\begin{equation}\label{cf3}
2M^2\Phi(1,\,2,\,M+1)
=2M^2\zeta(2,\,M+1)=b_0+
\K_{n=1}^{\infty}\frac{1}{b_n}.
\end{equation}
\end{theorem}

\begin{theorem}\label{conj_4}
Let $M\in\mathbb C$ satisfy $\Re M >0$, and let
\begin{align*}
b_0=M-1,\qquad
b_{2n-1}=\frac{2M}{n},\qquad
b_{2n}=\frac{M(2n+1)}{n(n+1)},\qquad n\in\mathbb N.
\end{align*}
Then
\begin{equation}\label{cf4}
2M^3\Phi(1,\,3,\,M+1)=2M^3\zeta(3,\,M+1)
=
b_0+\K_{n=1}^{\infty}\frac{1}{b_n}.
\end{equation}
\end{theorem}

For the selected parameters $M$, several familiar constants with their simple shifts can be obtained as special cases of Theorems~\ref{conj_1}--\ref{conj_4}. For example, setting $M=1$ in Theorem~\ref{conj_1} we get $2\Phi(-1,1,2)=2(1-\log 2)$ and consequently an explicit continued fraction representation of transcendental number $\log 2$. Similarly, Theorem~\ref{conj_2} with $M=1$ yields $2\Phi(-1,2,2)=2-\pi^2/6$.
If $M\in\mathbb{N}$ and $\Re s>1$, then \cite[Eq.~(25.11.4)]{DLMF}
\begin{align*}
\zeta(s,M+1)=\zeta(s)-H_M^{(s)},
\qquad \zeta(s)=\sum_{k=1}^{\infty}\frac{1}{k^s},\qquad
H_M^{(s)}=\sum_{k=1}^{M}\frac{1}{k^s}.
\end{align*}
Consequently, Theorems~\ref{conj_3} and~\ref{conj_4} also yield explicit continued fraction representations of the constant $\zeta(2)=\pi^2/6$ and Apéry's constant $\zeta(3)$, involving their finite harmonic-number shifts, cf. \cite[pp.~153--154]{BerndtRamanujanNotebooksII}, \cite{Gutnik2010}. There are a few more particular examples
\begin{align*}
&\Phi\left(1,2,\frac{3}{2}\right)
=\frac{\pi^2}{2}-4,\qquad
\Phi\left(1,2,\frac{4}{3}\right)
=\frac{2\pi^2}{3}
+3\sqrt{3}\,\operatorname{Cl}_2\left(\frac{2\pi}{3}\right)-9,\\
&\Phi\left(1,2,\frac{5}{4}\right)
=\pi^2+8G-16,\qquad
\Phi\left(1,3,\frac{3}{2}\right)=7\zeta(3)-8,
\end{align*}
where $\operatorname{Cl}_2$ denotes the Clausen function 
\[
\operatorname{Cl}_2(\theta)
=-\int_0^\theta
\log\left(2\sin\frac{t}{2}\right)\,dt,
\qquad 0\leq\theta\leq2\pi,\qquad
\operatorname{Cl}_2\left(\frac{2\pi}{3}\right)
=0.6766277376\ldots
\]
and $G=1-1/9+1/25-1/49+\cdots=0.915965\ldots$ is the Catalan constant.

\subsection{Origin of the main results}

The origin of the continued fraction in Theorem~\ref{conj_1} was as follows. The second author did certain numerical computations that involve the nontrivial zeros of the Riemann zeta function. One of the computed numbers, denote it by $\nu$, turned out to be close to the value of $2M\Phi(-1,\,1,\,M+1)$ with a lucky choice of $M=\mathrm{LCM}(1,\dots,10)$. In this case, an initial fragment of the classical ``arithmetical'' continued fraction (i.\,e., with partial numerators equal to $1$ and partial denominators being natural numbers) for $\nu$ coincided with an initial fragment of the \eqref{cf1}. For more details, 
see \cite{Matiyasevich} and \cite[Sec.~A.5]{Zagier} where the underlying continued fraction is related to a continued fraction for a Bernoulli-number generating function. The continued fractions in Theorems~\ref{conj_2}--\ref{conj_4} were discovered numerically by analogy with Theorem \ref{conj_1}.

\subsection{Prior work}
We first recall the standard relations between the Lerch transcendent,
the Hurwitz zeta function, and the polygamma functions. For $k\in\mathbb{N}$, it can be shown, see \cite[Eq.~(25.11.12)]{DLMF}, that
\begin{align}\label{hurwitz_to_polygamma}
\Phi(1,\,k+1,\,M)=\zeta(k+1,M)
=
\frac{(-1)^{k+1}}{k!}\psi_k(M),
\quad
M\in\mathbb{C}\notin\{0,-1,-2,\ldots\},
\end{align}
where $\psi_k$ denotes the polygamma function of order $k$, i.e. the $k+1$'th derivative of $\log\Gamma(s)$. Also, by \cite[Sec.~25.14]{DLMF} and \cite[Sec.~5.15]{DLMF},
\begin{equation}\label{lerch_polygamma_difference}
\Phi(-1,k,M+1)
=
\frac{(-1)^k}{2^k(k-1)!}
\left(
\psi_{k-1}\left(\frac{M+1}{2}\right)
-
\psi_{k-1}\left(\frac{M+2}{2}\right)
\right),
\end{equation}
where $M\in\mathbb{C}\notin\{-1,\,-2,\,\ldots\}$.

Various continued fraction identities of the right-hand side of \eqref{hurwitz_to_polygamma} and \eqref{lerch_polygamma_difference} are reported in the encyclopedic dictionary \cite{Cohen2026}, where relations to Theorems \ref{conj_1}--\ref{conj_4} are the following:
\begin{itemize}
\item the tail of the continued fraction in Theorem \ref{conj_1} is equivalent to the parametric family associated with entry 6.10.5, by setting $z=M+1$, $u=1$, and $k=-1$;
\item the even part (see Lemma \ref{even_part}) of the continued fraction in Theorem \ref{conj_2} is equivalent to the continued fraction represented by entry 6.12.2 with $z=M+1$;
\item the continued fraction in Theorem \ref{conj_3} is equivalent to the continued fraction in entry 6.11.3 with $z=M$;
\item the continued fraction in Theorem \ref{conj_4} is equivalent to the continued fraction in entry 6.13.2 with $z=M$.
\end{itemize}
However, these entries provide the corresponding continued-fraction identities but do not provide the proofs, references, or the particular parameterizations used in Theorems \ref{conj_1}--\ref{conj_4}. In addition, the tail of the continued fraction appearing in Theorem \ref{conj_1} is closely related to the continued fraction in \cite[Eq. (A.46)]{Zagier}. Indeed, setting $X=2M,\,M>0$ there, the tail beginning with the third denominator in Theorem \ref{conj_1} coincides, after reindexing, with the corresponding tail in \cite[Eq. (A.46)]{Zagier}. Thus, the continued fraction in Theorem \ref{conj_1} can be obtained from the one in \cite[Eq. (A.46)]{Zagier} by a modification of its initial terms. Although a proof of the identity in \cite[Eq. (A.46)]{Zagier} is not provided there, the reference to \cite[Problem 5.6(a)]{Lando} provides a derivation: after the substitution \(s=i/X\) in the continued fraction for the Bernoulli generating function and an equivalence transformation, one obtains the continued fraction in \cite[Eq. (A.46)]{Zagier}. Also, the continued-fraction representation in Theorem \ref{conj_1} for $M>0$ appears in unpublished notes by Natalia Korskova back in 2014. However, the proof of Theorem \ref{conj_1} given in this paper differs essentially from the described ones, contains all the necessary details, and extends the underlying identity for $\Re M>0$. The proof of Theorem \ref{conj_2}, in turn, provides a direct derivation of the stated continued fraction and establishes its convergence under the same condition $\Re M>0$.

Although in Theorems \ref{conj_1}--\ref{conj_4} we prove only four explicit representations, it is possible to obtain further continued fraction representations for the Lerch
transcendent with other choices of the parameters. Our proofs of Theorems \ref{conj_3} and \ref{conj_4}, unlike those of Theorems \ref{conj_1} and \ref{conj_2}, rely on the known (proved) continued fraction representations of the polygamma function. Because of \eqref{hurwitz_to_polygamma}, the continued fraction representations of $\psi_k$, obtained through their Stieltjes-transform representations \cite{ShentonBowman1971}, \cite{Shenton1983}, yield continued fraction representations of the Hurwitz zeta-function. Although explicit closed-form formulas are available for the
continued fractions of the trigamma $\psi_1$ and
tetragamma $\psi_2$ functions, no comparable closed-form formulas
are known to us for the coefficients of the corresponding continued
fractions for $k\geq3$. For more information, see
\cite[pp.~229--232]{Handbook} and \cite{BonanHamadaJones2005}. 

\section{Preliminaries for the proofs}\label{preliminaries}

In this section, we collect the necessary statements about continued fractions and hypergeometric functions that are used in the proofs of Theorems \ref{conj_1}--\ref{conj_4}.

\subsection{Continued fractions}

We begin with the canonical even and odd parts of a continued fraction. If the continued fraction
\[
b_0+\K\!\left(\frac{a_n}{b_n}\right)
\]
has convergents $f_n,\, n\in\mathbb{N}_0$, then the convergents of its even and odd parts, respectively, are
\[
g_n:=f_{2n},\qquad
h_n:=f_{2n+1},\qquad n\in\mathbb{N}_0.
\]
The following two classical lemmas describe the canonical even and odd parts of a continued fraction.
\begin{lem}[Lorentzen\&Waadeland (2008)]\label{even_part}
The continued fraction $b_0+\k(a_n/b_n)$ has an even part if and only if $b_{2n}\neq0$ for all $n\in\mathbb{N}_0$. Its canonical even part is then given by
\[
b_0+\frac{b_2a_1}{b_2b_1+a_2}\-\frac{a_2a_3b_4/b_2}{a_4+b_3b_4+a_3b_4/b_2}\-\frac{a_4a_5b_6/b_4}{a_6+b_5b_6+a_5b_6/b_4}\-\cds
\]
\end{lem}

\begin{proof}
Theorem 2.19 in \cite{Lorentzen}.
\end{proof}

\begin{lem}[Lorentzen\&Waadeland (2008)]\label{odd_part}
The continued fraction $b_0+\k(a_n/b_n)$ has an odd part if and only iff $b_{2n+1}\neq0$ for all $n\in\mathbb{N}_0$. Its canonical odd part is then given by
\begin{align*}
&\frac{a_1}{b_1}-\frac{a_1a_2b_3/b_1}{b_1(a_3+b_2b_3)+a_2b_3}\-\frac{a_3a_4b_5b_1/b_3}{a_5+b_4b_5+a_4b_5/b_3} \\
&\-\frac{a_5a_6b_7/b_5}{a_7+b_6b_7+a_6b_7/b_5}\-\frac{a_7a_8b_9/b_7}{a_9+b_8b_9+a_8b_9/b_7}\-\cds
\end{align*}
\end{lem}

\begin{proof}
Theorem 2.20 in \cite{Lorentzen}.
\end{proof}

Another important operation involving continued fractions is the equivalence transformation, which leaves all convergents unchanged.

\begin{lem}[Lorentzen \& Waadeland (2008)]\label{equiv_cf}
The continued fractions $\k(a_n/b_n)$ and $\k(c_n/d_n)$ are equivalent, that is,
\[
\frac{A_n}{B_n}=\frac{C_n}{D_n},
\qquad n\in\mathbb N,
\]
if and only if there exists a sequence $(r_n)_{n\geqslant0}$ of nonzero complex numbers with $r_0=1$ such that
\[
c_n=r_nr_{n-1}a_n,
\qquad
d_n=r_nb_n,
\qquad n\in\mathbb N.
\]
\end{lem}
\begin{proof} Theorem 2.14 in \cite{Lorentzen}. \end{proof}

The following is one of the main convergence criteria for continued fractions that we will use.

\begin{lem}[Van Vleck's Theorem, Lorentzen \& Waadeland (2008)]
\label{cf_convergence_1/bn}
Let $0<\varepsilon<\pi/2$, and suppose that
\[
b_n\in G_\varepsilon
:=
\left\{w\in\mathbb{C}:|\arg w|\leqslant
\frac{\pi}{2}-\varepsilon\right\}\cup\{0\},
\qquad n\in\mathbb{N},
\]
with $b_1\neq0$. Then the continued fraction $\k(1/b_n)$
converges if and only if
\[
\sum_{n=1}^{\infty}|b_n|=\infty.
\]
\end{lem}

\begin{proof}
This is Van Vleck's Theorem 3.37 in \cite{Lorentzen}.
\end{proof}

In some instances, the value of a continued fraction $\k(a_n/b_n)$ is determined by studying the associated three-term recurrence
\begin{align}\label{recurrence}
    x_n-b_nx_{n-1}-a_nx_{n-2}=0,\qquad n\in\N.
\end{align}

A sequence $(x_n)_{n\geqslant -1}$ satisfying \eqref{recurrence} is
called a solution of the recurrence. A nonzero solution $(x_n)$ of \eqref{recurrence} is called minimal if, for every linearly independent solution $(y_n)$, one has
\[
\frac{x_n}{y_n}\to0,
\qquad n\to\infty.
\]

The following classical result, due to Pincherle, establishes the connection between continued fractions and minimal solutions.

\begin{lem}[Pincherle (1894)]\label{Pincherle}
The continued fraction $\k(a_n/b_n)$ converges if and only if \eqref{recurrence} possesses a minimal solution $(x_n)_{\geq -1}$ with $x_0\neq0$. If the continued fraction below converges, then
\[
-\frac{x_{n-1}}{x_{n-2}}
=
\frac{a_n}{b_n}\+\frac{a_{n+1}}{b_{n+1}}\+\cds,
\qquad n\in\N,
\]
provided $x_n\neq0$ for all $n\in\mathbb{N}_0$.
\end{lem}

\begin{proof}
See \cite[Thm. 1.1]{Gautschi1967}.
\end{proof}

\subsection{Hypergeometric functions}

The generalized hypergeometric function is defined by
\begin{align}\label{ghf}
{}_pF_q\!\left(\begin{matrix}
a_1,\ldots,a_p\\
b_1,\ldots,b_q
\end{matrix};z\right)
=
\sum_{n=0}^{\infty}
\frac{(a_1)_n\cdots(a_p)_n}
{(b_1)_n\cdots(b_q)_n}
\frac{z^n}{n!},
\end{align}
where $b_j\notin\{0,-1,-2,\ldots\}$, $j=1,\,2,\,\ldots,\,q$, 
and
\[
(a)_n=\frac{\Gamma(a+n)}{\Gamma(a)}
\]
denotes the Pochhammer symbol.

It is known that the generalized hypergeometric series \eqref{ghf} converges absolutely for $|z|<1$. Moreover, it converges at $z=1$ if
\[
\Re\!\left(\sum_{j=1}^{q}b_j-\sum_{i=1}^{p}a_i\right)>0,
\]
and at $z=-1$ if
\[
\Re\!\left(\sum_{j=1}^{q}b_j-\sum_{i=1}^{p}a_i+1\right)>0.
\]
For these and other properties of generalized hypergeometric functions, see \cite{Bailey}.
In particular, we write
\begin{align}\label{def_F21}
{}_2F_1(a,b;c;z)
=
\sum_{n=0}^{\infty}
\frac{(a)_n(b)_n}{(c)_n}\frac{z^n}{n!}.
\end{align}
It is known \cite[Sec.~1.2]{Slater} that $y={}_2F_1(a,b;c;z)$
satisfies the differential equation
\begin{align}\label{2F1_diff_eq}
z(1-z)\frac{d^2y}{dz^2}
+\bigl(c-(a+b+1)z\bigr)\frac{dy}{dz}
-aby=0,
\end{align}
while the generalized hypergeometric function \eqref{ghf}, $w={}_pF_q(a_1,\ldots,a_p;\,b_1,\ldots,b_q;\,z)$,
satisfies (see \cite[Sec.~2.1.2]{Slater})
\begin{align}\label{pFq_diff_eq}
\left(
z\frac{d}{dz}
\prod_{j=1}^{q}
\left(z\frac{d}{dz}+b_j-1\right)
-
z
\prod_{i=1}^{p}
\left(z\frac{d}{dz}+a_i\right)
\right)w=0.
\end{align}

We also use Euler's integral representation. In particular,
for $\Re b>0$ and $\Re(c-b)>0$, one has (see \cite[Sec.~1.6]{Slater})
\begin{align}\label{2F1_integral_form}
{}_2F_1(a,b;c;z)
&=
\frac{\Gamma(c)}
{\Gamma(b)\Gamma(c-b)}
\int_0^1
t^{b-1}(1-t)^{c-b-1}(1-zt)^{-a}\,dt.
\end{align}
More generally, for $\Re c>0$ and $\Re(d-c)>0$ (see \cite[Eq.~(4.1.2)]{Slater}),
\begin{align}\label{pFq_integral_form}\nonumber
{}_{p+1}F_{q+1}\left(
\begin{matrix}
a_1,\ldots,a_p,c\\
b_1,\ldots,b_q,d
\end{matrix}
;z
\right)
&=
\frac{\Gamma(d)}
{\Gamma(c)\Gamma(d-c)}
\\
&\quad\times
\int_0^1
t^{c-1}(1-t)^{d-c-1}
{}_pF_q\left(
\begin{matrix}
a_1,\ldots,a_p\\
b_1,\ldots,b_q
\end{matrix}
;zt
\right)\,dt.
\end{align}
We shall also use the following transformation due to Thomae \cite[Sec.~3.2, Eq.~(1)]{Bailey}. Let
\begin{align*}
s=e+f-a-b-c.
\end{align*}
If $\Re(a)>0$ and $\Re(s)>0$, then
\begin{align}\label{thomae}
\Fhyp{a,b,c}{e,f}{1}
&=
\frac{\Gamma(e)\Gamma(f)\Gamma(s)}
{\Gamma(a)\Gamma(s+b)\Gamma(s+c)}
\Fhyp{e-a,f-a,s}{s+b,s+c}{1}.
\end{align}
To study the growth of hypergeometric functions, we employ the following gamma-ratio asymptotic.
\begin{lem}\label{gamma_ratio_asymptotic}
Let $\alpha,\beta\in\mathbb{C}$ be fixed. Then
\[
\frac{\Gamma(j+\alpha)}{\Gamma(j+\beta)}
\sim j^{\alpha-\beta},
\qquad
j\to\infty.
\]
In particular,
\[
\left|
\frac{\Gamma(j+\alpha)}{\Gamma(j+\beta)}
\right|
=
O\!\left(j^{\Re(\alpha-\beta)}\right),
\qquad
j\to\infty.
\]
\end{lem}

\begin{proof}
The first assertion is the standard gamma-ratio asymptotic
\cite[Sec.~5.11]{DLMF}. The second assertion follows immediately by
taking absolute values.
\end{proof}

\section{Proof of the first theorem}\label{sec1}

We first formulate and prove several auxiliary statements that are needed to prove Theorem \ref{conj_1}.

\subsection{Auxiliary results for the proof of Theorem \ref{conj_1}}

\begin{lem}\label{lerch_to_hyp}
If $\Re M>0$, then
\[
2M\Phi(-1,1,M+1)
=
\frac{2M}{M+1}\,
{}_2F_1(1,M+1;M+2;-1).
\]
\end{lem}

\begin{proof}
Using the identity $\Gamma(s+1)=s\Gamma(s)$ and
\[
(1)_n=n!,\qquad \frac{(M+1)_n}{(M+2)_n}
=
\frac{\Gamma(M+n+1)\Gamma(M+2)}
{\Gamma(M+1)\Gamma(M+n+2)}
=
\frac{M+1}{M+n+1},
\]
we obtain
\begin{align*}
{}_2F_1(1,M+1;M+2;-1)
&=
\sum_{n=0}^{\infty}
\frac{(1)_n(M+1)_n}{(M+2)_n}
\frac{(-1)^n}{n!} \\
&=
(M+1)
\sum_{n=0}^{\infty}
\frac{(-1)^n}{M+n+1}
=
(M+1)\Phi(-1,1,M+1).
\end{align*}
\end{proof}

\begin{lem}\label{2F1_ratio}
Let $a,b\in\C$, $c\in\C\setminus\Z_0^-$, and $\Re z<1/2$, where $\Z_0^-=\{0,-1,-2,\ldots\}$.
Then the continued fraction below converges and
\[
\frac{{}_2F_1(a,b;c;z)}
     {{}_2F_1(a+1,b+1;c+1;z)}
=
\frac{c-(a+b+1)z}{c}
+
\frac{1}{c}
\K_{n=1}^{\infty}
\frac{c_n(z-z^2)}{e_n+d_nz},
\]
where
\[
c_n=(a+n)(b+n),\qquad
d_n=-(a+b+2n+1),\qquad
e_n=c+n,\qquad n\in\mathbb{N}.
\]
\end{lem}

\begin{proof}
See \cite[(15.3.13a)--(15.3.13b)]{Handbook}. The convergence convention for continued-fraction representations used there is stated on p.~4 of \cite{Handbook}.
\end{proof}

\begin{lem}\label{cf_evaluated}
If $\Re M>0$, then the continued fraction below converges and
\begin{equation*}
\K_{n=1}^{\infty} \frac{-(n+2)(2n+2M)}{3n+2M+5}=(M+2)\frac{_2F_1(2,\,M,\,M+2,\,-1)}{_2F_1(3,\,M+1,\,M+3,\,-1)}-(2M+5).
\end{equation*}
\end{lem}

\begin{proof}
Apply Lemma \ref{2F1_ratio} with $a=2,\,b=M,\,c=M+2,\,z=-1$.
\end{proof}

\begin{lem}\label{lem:BM}
Let $a,b,c,\alpha,\beta\in\mathbb{C}$, with $\beta\neq0$ and $\beta(\alpha+\beta)=1$,
and define
\[
\mu_m:=
\K_{n=1}^{\infty}
\frac{(n+a+m)(n+b)}
{\alpha n+c+m(\alpha+\beta)},
\qquad m\in\mathbb{N}_0.
\]
Assume that, for every $m\in\mathbb{N}_0$,
\begin{align*}
\beta(a+m+1)+c+(m-b)(\alpha+\beta)\neq0,
\end{align*}
and that, for every $m\in\mathbb{N}_0$ and $n\in\mathbb{N}$, $n+a+m\neq0$.

Assume further that the continued fractions $\mu_m$ converge for all
$m\in\mathbb{N}_0$ and denote $\operatorname{BM}(\mu_m)$ the Bauer--Muir transformation of $\mu_m$. Then
\begin{align*}
\operatorname{BM}(\mu_m)
=
\beta(a+m+1)
+
\frac{\theta_{m+1}}
{\delta_{m+1}+\mu_{m+1}},
\end{align*}
where
\[
\theta_{m+1}
=
-\beta(a+m+1)
\Bigl(
\beta(a+m+1)+c+(m-b)(\alpha+\beta)
\Bigr)
\]
and
\[
\delta_{m+1}
=
\alpha+c+m(\alpha+\beta)+\beta(a+m+2).
\]
\end{lem}

\begin{proof}
Fix $m\in\mathbb{N}_0$ and write
\[
a_n=(n+a+m)(n+b),
\qquad
b_n=\alpha n+c+m(\alpha+\beta).
\]
Apply the Bauer--Muir transformation
\cite[Thm.~2.18]{Lorentzen}
with the modifying sequence
\begin{align}\label{modifying_sequence}
w_n=\beta(n+a+m+1),\qquad n\in\N_0.
\end{align}
The quantities $\lambda_n$ in the Bauer--Muir transformation \cite[Eq. (2.3.2)]{Lorentzen} are
\[
\lambda_n=a_n-w_{n-1}(b_n+w_n).
\]
Using $\beta(\alpha+\beta)=1$, we obtain
\begin{align*}
\lambda_n
&=(n+a+m)
\Bigl[
n+b
-\beta\bigl(
\alpha n+c+m(\alpha+\beta)
+\beta(n+a+m+1)
\bigr)
\Bigr]\\
&=(n+a+m)
\Bigl[
(1-\beta(\alpha+\beta))n
+b-\beta c
-\beta^2(a+m+1)
-m\beta(\alpha+\beta)
\Bigr]\\
&=-\beta(n+a+m)
\Bigl(
\beta(a+m+1)+c+(m-b)(\alpha+\beta)
\Bigr).
\end{align*}
By the assumptions, $\lambda_n\neq0$ for every $n\in\N$. Consequently,
\begin{align*}
\frac{\lambda_{n+1}}{\lambda_n}
&=
\frac{n+a+m+1}{n+a+m},\\
w_{n-1}\frac{\lambda_{n+1}}{\lambda_n}
&=
\beta(n+a+m+1),\\
c_{n+1}:=a_n\frac{\lambda_{n+1}}{\lambda_n}
&=
(n+a+m+1)(n+b),
\end{align*}
and
\begin{align*}
&d_{n+1}:=b_{n+1}+w_{n+1}
-w_{n-1}\frac{\lambda_{n+1}}{\lambda_n}\\
&\quad=
\alpha(n+1)+c+m(\alpha+\beta)
+\beta(n+a+m+2)
-\beta(n+a+m+1)\\
&\quad=
\alpha n+c+(m+1)(\alpha+\beta).
\end{align*}
Therefore, by \cite[Eq. (2.3.3)]{Lorentzen},
with
\begin{align*}
w_0&=\beta(a+m+1),\\
\lambda_1
&=
-\beta(a+m+1)
\Bigl(
\beta(a+m+1)+c+(m-b)(\alpha+\beta)
\Bigr)
=\theta_{m+1},\\
b_1+w_1
&=
\alpha+c+m(\alpha+\beta)+\beta(a+m+2)
=\delta_{m+1}.
\end{align*}
we obtain
\[
\operatorname{BM}(\mu_m)
=w_0+
\frac{\lambda_1}
{b_1+w_1+\mu_{m+1}}
.
\]
\end{proof}
\begin{lem}\label{cf_transformed}
Let $M\in\mathbb{C}$ satisfy $\Re M>0$. Then
\begin{equation}
\K_{n=1}^{\infty}
\frac{(n+1)(n+2)}{2M}
=
3+
\K_{n=1}^{\infty}
\frac{-2(n+2)(n+M)}
{3n+2M+5}.
\label{cf_transformed_identity}
\end{equation}
\end{lem}

\begin{proof}
We first prove the identity for $M>0$. Lemma \ref{lem:BM} with $\alpha=0$, $\beta=1$, $a=2$, $b=1$, $c=2M$ yields the positive modifying sequence $w_n=n+m+3$ \eqref{modifying_sequence} and
\begin{align}\label{mu_def}
\mu_m&=\K_{n=1}^{\infty}\frac{(n+m+2)(n+1)}{2M+m},\\
\operatorname{BM}(\mu_m)
&=
m+3+
\frac{-2(m+3)(m+M+1)}
{2M+2m+4+\mu_{m+1}},\quad m\in\mathbb{N}_0,
\label{to_iterate}
\end{align}
where, by Lemma \ref{cf_convergence_1/bn}, $\mu_m$ converges for all $m\in\mathbb{N}_0$. Since the continued fraction defining $\mu_m$ is positive, by the argument \cite[p. 83]{Lorentzen}, see also \cite[p.~27]{Perr57}), the Bauer--Muir transformation has the same value as $\mu_m$, i.e. $\operatorname{BM}(\mu_m)=\mu_m$. Therefore, applying the identity \eqref{to_iterate} successively for $m=0,1,\ldots,N-1$, we obtain 
\begin{align}\label{up_to_N}
\mu_0
=
3+
\cfrac{
-6(M+1)
}{
2M+8+
\cfrac{
-8(M+2)
}{
2M+11+
\ddots+
\cfrac{
-2(N+2)(N+M)
}{
2M+2N+2+\mu_N
}
}
}.
\end{align}

By denoting $T_m:=\mu_m-(m+3)$, $m\in\mathbb{N}_0$, from \eqref{to_iterate}, we obtain
\begin{align*}
T_{m-1}=\frac{-2(m+2)(m+M)}{2M+3m+5+T_{m}},\qquad m\in\mathbb{N}.
\end{align*}
Define
\[
S_n(x):=
\frac{-2(n+2)(n+M)}
{3n+2M+5+x}, \qquad x\geqslant-n-3, \qquad n\in\mathbb{N},\qquad x\in\mathbb{R}.
\]
Then
\[
T_{n-1}=S_{n}(T_n), \qquad n\in\mathbb{N}.
\]
Applying $T_{n-1}=S_n(T_n)$, for $n=1,\,2,\,\ldots,\,N$, we obtain
\begin{align*}
\mu_0-3
=S_1\circ S_2\circ\cdots\circ S_N(T_N)
=
\cfrac{
-6(M+1)
}{
2M+8+
\cfrac{
-8(M+2)
}{
2M+11+
\ddots+
\cfrac{
-2(N+2)(N+M)
}{
2M+3N+5+T_N
}
}
}.
\end{align*}
On the other hand, the convergents $C_N$, $N\in\N$ of the continued fraction
\begin{align*}
3+
\K_{n=1}^{\infty}
\frac{-2(n+2)(n+M)}
{3n+2M+5}
\end{align*}
are
\begin{align*}
C_1=3+S_1(0),\qquad C_2=3+S_1(S_2(0)), \qquad \ldots, \qquad C_N=
3+S_1\circ S_2\circ\cdots\circ S_N(0).
\end{align*}
It therefore remains to show that $C_N\to\mu_0$, $N\to\infty$. If $x\geqslant-n-3$, $x\in\mathbb{R}$, $n\in\mathbb{N}$, it follows
\begin{align*}
3n+2M+5+x\geqslant2(n+M+1)>2(n+M)>0.
\end{align*}
Consequently
\begin{align*}
-(n+2)<\frac{-2(n+2)(n+M)}{3n+2M+5+x}=S_n(x)<0,\qquad n\in\mathbb{N}.
\end{align*}
and
\begin{align*}
&0<S_n'(x)
=
\frac{2(n+2)(n+M)}
{(3n+2M+5+x)^2}
\leqslant
\frac{(n+2)(n+M)}
{4(n+M+1)^2}
<
\frac{n+2}{4(n+M+1)}
\\
&
<\frac{n+2}{4(n+1)}\leqslant\frac{3}{8}
<\frac12,\quad n\in\mathbb{N}.
\end{align*}
By the mean value theorem, for any
$x,y\in[-n-3,0]$ there exists $\xi$ between $x$ and $y$ such that
\[
|S_n(x)-S_n(y)|
=
|S_n'(\xi)|\,|x-y|
<
\frac12|x-y|, \qquad n\in\mathbb{N}.
\]
By \eqref{mu_def}, for all $m\in\mathbb{N}_0$ and $M>0$,
\begin{align*}
0<\mu_m=\frac{2m+6}{m+2M+\K_{n=2}^{\infty}\frac{(n+m+2)(n+1)}{2M+m}}<\frac{2m+6}{m+2M}.
\end{align*}
Therefore
\[
-N-3<T_N<-N-3+\frac{2N+6}{N+2M},\qquad N\in\mathbb{N},
\]
where $T_N=\mu_N-(N+3)$. If, in addition, $N>2(1-M)$, then $-N-3<T_N<0$. Thus, applying the mean value theorem iteratively $N$ times, with
\begin{align*}
S_n([-n-3,0])\subset(-n-2,0)\subset[-n-2,0],
\qquad n\in\mathbb N,
\end{align*}

we obtain
\begin{align*}
&|\mu_0-C_N|=|S_1\circ \cdots \circ S_N(T_N)-S_1\circ\cdots\circ S_N(0)|\\
&\leqslant\frac{1}{2}
|S_2\circ \cdots \circ S_N(T_{N})-S_2\circ\cdots\circ S_{N}(0)|\leqslant \ldots\\
&\leqslant
\frac{1}{2^{N-1}}|S_N(T_N)-S_N(0)|\leqslant\frac{|T_N|}{2^N}\to0,\,N\to\infty.
\end{align*}


It remains to extend the identity \eqref{cf_transformed_identity} to
$
H:=\{M\in\mathbb{C}:\Re M>0\}.
$
Consider
\[
\mu_0(M)=
\K_{n=1}^{\infty}
\frac{a_n}{b_n}, \qquad a_n=(n+2)(n+1),\qquad b_n=2M.
\]
Define
\[
r_0=1, \qquad r_nr_{n-1}a_n=1,\qquad n\in\mathbb N.
\]
Then, by Lemma \ref{equiv_cf}, the continued fraction defining
$\mu_0$ is equivalent to
\[
\K_{n=1}^{\infty}\frac{1}{2M r_n}.
\]
Since the sequence $r_n>0$, $n\in\mathbb{N}$, $r_n$ does not depend on $M$, and
\[
r_{2k}
=
\prod_{j=1}^k
\frac{2j}{2j+2}
=
\frac{1}{k+1}
\]
it follows
\[
\sum_{n=1}^{\infty}r_n=\infty.
\]
Consequently, there exists $\varepsilon\in(0,\pi/2)$ such that
\[
2Mr_n\in G_\varepsilon=\left\{w\in\mathbb{C}:|\arg w|\leqslant
\frac{\pi}{2}-\varepsilon\right\}\cup\{0\},\qquad n\in\mathbb N.
\]
Since
\[
\sum_{n=1}^{\infty}|2Mr_n|
=
2|M|\sum_{n=1}^{\infty}r_n
=\infty,
\]
Lemma \ref{cf_convergence_1/bn} implies that $\mu_0(M)$ converges
for every $M\in H$.

We next show that $\mu_0$ is holomorphic in $H$. Let $K\subset H$
be compact. Since
\[
\inf_{M\in K}\Re M>0,
\]
there exists an open neighborhood $U$ of $K$ with
$\overline U\subset H$ such that
\[
\Re(2Mr_n)>0,\qquad M\in U,\qquad n\in\mathbb N.
\]
Then, each finite convergent of $\k(1/2Mr_n)$ belongs to the open right half-plane. Indeed,
the maps
\[
w\mapsto \frac{1}{2Mr_n+w},\qquad n\in\mathbb{N}.
\]
map the closed right half-plane into the open right half-plane. Since these finite convergents converge for every $M\in U$, the
Stieltjes--Vitali theorem \cite[Thm.~3.10, p.~115]{Lorentzen}
implies that the convergence is locally uniform in $U$ and that its
limit is holomorphic. Hence $F(M):=\mu_0-3$ is holomorphic in $H$.

On the other hand, by Lemma \ref{cf_evaluated}, the continued
fraction
\[
G(M):=
\K_{n=1}^{\infty}
\frac{-2(n+2)(n+M)}
     {3n+2M+5}
\]
converges for every $M\in H$ and is represented there by the
hypergeometric expression given in that lemma. Hence $G(M)$ is
holomorphic in $H$. For $M>0$, the first part of the proof established $F(M)=G(M)$. Since $F$ and $G$ are holomorphic in the connected domain $H$ and
$(0,\infty)\subset H$ has an accumulation point in $H$, the identity
theorem \cite[Thm.~6.5]{Berg} yields
\[
F(M)=G(M),\qquad M\in H.
\]
\end{proof}

\begin{lem}\label{2F1_relations}
Let $\Re M>0$. Then
\begin{align}
{}_2F_1(2,M;M+2;-1)
&=
M(2M-1)\,{}_2F_1(1,M+1;M+2;-1)-M^2+1,
\label{hyp_2}\\
{}_2F_1(3,M+1;M+3;-1)
&=
M^2(M+2)\,{}_2F_1(1,M+1;M+2;-1)
\nonumber\\
&\quad
+\frac{(M+1)(M+2)(1-2M)}{4}.
\label{hyp_3}
\end{align}
\end{lem}

\begin{proof}
To prove \eqref{hyp_2}, we apply (1.8) and (1.17) from
\cite{Rakha} with $a=1$,  $b=M$, $c=M+2$, $z=-1$, which respectively give 
\begin{align}
{}_2F_1(2,M;M+2;-1)
&=
\frac{1-2M}{2}\,{}_2F_1(1,M;M+2;-1)
+\frac{M+1}{2},
\label{hyp_aux1}\\
(1-M){}_2F_1(1,M;M+2;-1)
&=
{}_2F_1(2,M;M+2;-1)
\nonumber\\
&\hspace{2.2cm}
-M\,{}_2F_1(1,M+1;M+2;-1).
\label{hyp_aux2}
\end{align}
Substituting \eqref{hyp_aux1} into \eqref{hyp_aux2} and collecting
the terms containing ${}_2F_1(1,M;M+2;-1)$, we obtain
\[
{}_2F_1(1,M;M+2;-1)
=
M+1-2M\,{}_2F_1(1,M+1;M+2;-1).
\]
Substitution of this identity into \eqref{hyp_aux1} yields
\eqref{hyp_aux2}.

To prove \eqref{hyp_3}, we use relations (1.24) and (1.31) from
\cite{Rakha}, with $a=2$, $b=M+1$, $c=M+3$, and $z=-1$, which
respectively give
\begin{align}
{}_2F_1(3,M+1;M+3;-1)
&=-\frac{M}{2}{}_2F_1(2,M+1;M+3;-1)\notag\\
&\quad+\frac{M+2}{2}{}_2F_1(2,M+1;M+2;-1),
\label{hyp_aux3}\\
{}_2F_1(2,M+1;M+3;-1)
&=2(M+2){}_2F_1(2,M+1;M+2;-1)\notag\\
&\quad-(M+2){}_2F_1(1,M+1;M+2;-1).
\label{hyp_aux4}
\end{align}
Relation (1.8) from \cite{Rakha}, with $a=1$, $b=M+1$, $c=M+2$,
and $z=-1$, yields
\begin{equation}\label{hyp_aux5}
{}_2F_1(2,M+1;M+2;-1)
=
-M\,{}_2F_1(1,M+1;M+2;-1)+\frac{M+1}{2}.
\end{equation}
Substituting \eqref{hyp_aux4} into \eqref{hyp_aux3}, we obtain
\begin{align*}
{}_2F_1(3,M+1;M+3;-1)
&=\frac{(M+2)(1-2M)}{2}
  {}_2F_1(2,M+1;M+2;-1)\\
&\quad+\frac{M(M+2)}{2}
  {}_2F_1(1,M+1;M+2;-1).
\end{align*}
Finally, substituting \eqref{hyp_aux5} into the latter identity gives
\[
{}_2F_1(3,M+1;M+3;-1)
=
M^2(M+2){}_2F_1(1,M+1;M+2;-1)
+\frac{(M+1)(M+2)(1-2M)}{4},
\]
which is \eqref{hyp_3}.
\end{proof}

\begin{lem}\label{lem:cf_convergence_from_tail}
Let $\k(a_n/b_n)$ be a continued fraction, and let $A_n$ and $B_n$
denote its $n$th canonical numerator and denominator, respectively.
If, for some $N\in\mathbb N$, the $N$th tail of $\k(a_n/b_n)$ converges to
$f^{(N)}$ and $B_{N-1}f^{(N)}+B_N\neq0$,
then $\k(a_n/b_n)$ converges.
\end{lem}

\begin{proof}
Fix $N$ such that the $N$th tail converges, and denote its $k$th
convergent by $f_k^{(N)}$. The relation between a continued
fraction and its tails \cite[Eq.~(1.3.2)]{Handbook} gives
\[
f_{N+k}
=
\frac{A_{N-1}f_k^{(N)}+A_N}
     {B_{N-1}f_k^{(N)}+B_N},
\qquad k\in\N.
\]
Since
\[
f_k^{(N)}\to f^{(N)}, \qquad k\to\infty
\]
and
\[
B_{N-1}f^{(N)}+B_N
\neq0,
\]
we obtain
\[
f_{N+k}
\to
\frac{A_{N-1}f^{(N)}+A_N}
     {B_{N-1}f^{(N)}+B_N},\qquad k\to\infty.
\]
Thus the sequence of convergents of $\k(a_n/b_n)$ converges, and hence
the continued fraction $\k(a_n/b_n)$ converges.
\end{proof}

\subsection{Proof of Theorem~\ref{conj_1}}

By the assumptions of Theorem~\ref{conj_1},
\begin{align*}\label{b_initial}
b_0=0,\qquad
b_1=1,\qquad
b_2=2M-1,\qquad
b_n=\frac{2M}{n-1},\qquad n\in \{3,\,4,\,\ldots\}.
\end{align*}
Since $\Re M>0$, there exists $\varepsilon\in(0,\pi/2)$ such that
\begin{align*}
|\arg M|\leqslant\frac{\pi}{2}-\varepsilon.
\end{align*}
Hence $b_n=2M/(n-1)\in G_\varepsilon$, $n\in \{3,\,4,\,\ldots\}$,
where
\[
G_\varepsilon
=
\left\{
w\in\mathbb{C}:
|\arg w|\leqslant\frac{\pi}{2}-\varepsilon
\right\}\cup\{0\}.
\]
Moreover,
\[
\sum_{n=3}^{\infty}|b_n|
=
2|M|\sum_{n=3}^{\infty}\frac{1}{n-1}
=\infty.
\]
Therefore, by Lemma~\ref{cf_convergence_1/bn}, the continued fraction
\[
T=\K_{n=3}^{\infty}\frac{1}{b_n}
\]
converges for every $M$ satisfying $\Re M>0$.

Notice that $T$ is the second tail of the continued fraction $\k(1/b_n)$. Denote by $B_n$ the canonical denominator of the continued fraction $\k(1/b_n)$. By \eqref{fundamental_recurrence}, $B_2/B_1=2M$. We next show that $T\neq -2M$.
Define the linear fractional transformation for the $n$-th step as
\begin{align*}
    s_n(w)=\frac{1}{b_n+w},\quad n\in \{3,\,4,\,\ldots\},
\end{align*}
where $\Re w\geq0$. Since $\Re b_n>0$ we have
\begin{align*}
\Re s_n(w)=\Re\frac{1}{b_n+w} = \frac{\Re (b_n+w)}{\abs{b_n+w}^2} > 0.
\end{align*}

The $n$th convergent $T_n$ of $T$ can be written as
\begin{align*}
T_n = s_3\circ s_4 \circ\ldots \circ s_{n+2}(0), \qquad n\in\mathbb{N}.
\end{align*}
Since 
\begin{align*}
T_1=s_3(0)=\frac{1}{b_3},\quad T_2=s_3(s_4(0))=\frac{1}{b_3+1/b_4}
\end{align*}
and $\Re T_1,\,\Re T_2>0$, by induction we have that $\Re T_n >0$ for all $n\in \N$. Letting $n\to\infty$, we conclude that $\Re T \geq 0$, which implies $T\neq -2M$. Therefore, by Lemma~\ref{lem:cf_convergence_from_tail}, the continued fraction $\k (1/b_n)$ converges for every $M$ satisfying $\Re M>0$. Denote its limit by $L$.

Applying the equivalence transformation in
Lemma~\ref{equiv_cf} with
\[
r_0=1,\qquad r_1=2M,\qquad r_2=1,\qquad
r_n=n-1,\quad n\in \{3,\,4,\,\ldots\},
\]
we obtain
\[
L=\K_{n=1}^{\infty}\frac{c_n}{d_n},
\]
where
\begin{align*}
c_1&=2M,\qquad
c_2=2M,\qquad
c_3=2,\qquad
c_n=(n-1)(n-2),\quad n\in \{4,\,5,\,\ldots\},\\
d_1&=2M,\qquad
d_2=2M-1,\qquad
d_n=2M,\quad n\in \{3,\,4,\,\ldots\}.
\end{align*}
Let
\[
K=
\K_{n=4}^{\infty}\frac{c_n}{d_n}
=
\K_{n=4}^{\infty}\frac{(n-1)(n-2)}{2M}
=
\K_{n=1}^{\infty}\frac{(n+2)(n+1)}{2M}.
\]
To prove the convergence of $K$, we apply the equivalence transformation
from Lemma~\ref{equiv_cf} with $r_0=1$ and
\[
r_n r_{n-1}(n+2)(n+1)=1,\qquad n\in\N.
\]
This gives
\[
r_{2n}=\frac{1}{n+1},
\qquad
r_{2n-1}=\frac{1}{4n+2},
\qquad n\in\N.
\]
Therefore $K$ is equivalent to $\k(1/\widetilde b_n)$,
where
\begin{align*}
\widetilde b_{2n-1}=\frac{M}{2n+1},
\qquad
\widetilde b_{2n}=\frac{2M}{n+1},\qquad n\in\mathbb{N}.
\end{align*}
Since $\Re M>0$, there exists $\varepsilon\in(0,\pi/2)$ such that
$\widetilde b_n\in G_\varepsilon$ for all $n\in\N$. Moreover,
\[
\sum_{n=1}^{\infty}|\widetilde b_n|
\geqslant
2|M|\sum_{n=1}^{\infty}\frac{1}{n+1}
=\infty.
\]
Therefore, by Lemma~\ref{cf_convergence_1/bn}, the transformed
continued fraction converges, and hence $K$ converges. Therefore
\begin{align}\label{L_express}
L
&=
\cfrac{2M}
{2M+
 \cfrac{2M}
 {2M-1+
  \cfrac{2}
  {2M+K}}}=
\frac{4M^2+2MK-2M-K+2}
     {4M^2+2MK+2}.
\end{align}
By Lemma~\ref{cf_transformed}
\begin{align}\label{transformation}
K
=
3+\K_{n=1}^{\infty}
\frac{-2(n+2)(n+M)}{3n+2M+5}.
\end{align}
Lemma~\ref{cf_evaluated} further gives
\begin{equation}\label{final_cf_evaluated}
K
=
(M+2)
\frac{{}_2F_1(2,M;M+2;-1)}
     {{}_2F_1(3,M+1;M+3;-1)}
-2(M+1).
\end{equation}
Substituting \eqref{final_cf_evaluated} into \eqref{L_express},
applying Lemma~\ref{2F1_relations}, and simplifying, we obtain
\[
L=
\frac{2M}{M+1}
\,{}_2F_1(1,M+1;M+2;-1).
\]
The statement of Theorem \ref{conj_1} now follows by applying Lemma~\ref{lerch_to_hyp}.
\qed

\section{Proof of the second theorem}

As in Section \ref{sec1}, we first formulate and prove several auxiliary statements that are needed to prove Theorem \ref{conj_2}.

\subsection{Auxiliary results for the proof of Theorem \ref{conj_2}}

\begin{lem}\label{recurrence_sol_lem}
Let $\Re M>0$, and for $k\in\mathbb{N}_0$ define
\begin{align}\label{def_y_k}
y_k=
\frac{k!}{(M)_{k+1}}
\Fhyp{M,\,M+1,\,M+1}{2M+2,\,M+k+1}{1},
\end{align}
where $(M)_{k+1}$ denotes the Pochhammer symbol. Then
\begin{equation}\label{recurrence_sol_eq}
k^2(y_{k+1}-2y_k+y_{k-1})=M(M+1)y_k,
\qquad k\in \N.
\end{equation}
\end{lem}

\begin{proof}
Applying \eqref{pFq_integral_form} with $p=2,\,q=1$ and
\[
a_1=M,\qquad a_2=c=M+1,\qquad
b_1=2M+2,\qquad d=M+k+1,\qquad z=1,
\]
and using
\begin{align*}
\Fhyp{M,\,M+1,\,M+1}{2M+2,\,M+k+1}{1}=\Fhyp{M+1,\,M+1,\,M}{2M+2,\,M+k+1}{1},
\end{align*}
we obtain
\begin{align*}
y_k
=\int_0^1 t^{M-1}(1-t)^k
{}_2F_1(M+1,M+1;2M+2;t)\,dt.
\end{align*}
Denote
\[
f(t)={}_2F_1(M+1,M+1;2M+2;t).
\]
Then
\begin{align*}
y_{k+1}-2y_k+y_{k-1}
&=
\int_0^1 t^{M-1}(1-t)^{k-1}
\bigl((1-t)^2-2(1-t)+1\bigr)f(t)\,dt\\
&=
\int_0^1 t^{M+1}(1-t)^{k-1}f(t)\,dt.
\end{align*}
Let us denote $u(t)=t^{M+1}f(t)$ and observe
\begin{align*}
k^2(1-t)^{k-1}
=
\frac{d^2}{dt^2}(1-t)^{k+1}
+
\frac{d}{dt}(1-t)^k.
\end{align*}
Then, upon integration by parts,
\begin{align*}
&k^2(y_{k+1}-2y_k+y_{k-1})=\int_{0}^{1}u(t)\,d\,\left(\left((1-t)^{k+1}\right)'+(1-t)^k\right)\\
&=-k (1-t)^k u(t)\Big|^1_0-
\int_{0}^{1}\left(\left((1-t)^{k+1}\right)'+(1-t)^k\right)u'(t)\,dt\\
&=-\left(k (1-t)^k u(t)+(1-t)^{k+1}u'(t)\right)\Big|^1_0
+\int_{0}^{1}(1-t)^k\left((1-t)u''(t)-u'(t)\right)\,dt=:B+I.
\end{align*}
Since $u'(t)=(M+1)t^Mf(t)+t^{M+1}f'(t)$, we have
\begin{align*}
B
=
-\left(
\left(
k(1-t)^k t^{M+1}
+(M+1)(1-t)^{k+1}t^M
\right)f(t)
+(1-t)^{k+1}t^{M+1}f'(t)
\right)\Big|_0^1.
\end{align*}
By definition \eqref{def_F21},
\begin{align*}
f(0)=1,\,f'(0)=\frac{M+1}{2},\,f(t)=O\left(\log \frac{1}{1-t}\right),\,
f'(t)=O\left(\frac{1}{1-t}\right),\,t\to1^-.
\end{align*}
Therefore, $B=0$ because $\Re M>0$ and $k\in \N$. By computing $u''(t)$ and inserting it into $I$, we obtain
\begin{align*}
&k^2(y_{k+1}-2y_k+y_{k-1})-M(M+1)y_k\\
&\quad=
\int_0^1 t^M(1-t)^k
\left(
t(1-t)f''(t)
+\bigl(2M+2-(2M+3)t\bigr)f'(t)
-(M+1)^2f(t)
\right)\,dt.
\end{align*}
The latter integral $=0$ by the hypergeometric
differential equation \eqref{2F1_diff_eq}, with
\[
a=b=M+1,\qquad c=2M+2,
\]
which proves \eqref{recurrence_sol_eq}.
\end{proof}

\begin{lem}\label{solution_asymptotics}
Let $\Re M>0$. Then
\begin{align}\label{F_32_k}
y_k=\frac{k!}{(M)_{k+1}} \Fhyp{M,\,M+1,\,M+1}{2M+2,\,M+k+1}{1} \sim \frac{\Gamma(M)}{k^M},
\qquad
k\to\infty.
\end{align}
\end{lem}

\begin{proof}
By the definition of the generalized hypergeometric function \eqref{ghf}, we study $_3F_2$ from \eqref{F_32_k}. Define
\begin{align*}
_3F_2(M,k)=F_k
:=
1+\sum_{j=1}^{\infty}
\frac{(M)_j(M+1)_j^2}
     {(2M+2)_j(M+k+1)_j\,j!}.
\end{align*}
Since $\Re M>0$, for every $r\geqslant0$ and $k\in\N_0$,
\[
|M+k+1+r|^2-|M+1+r|^2
=
k\bigl(2(\Re M+1+r)+k\bigr)
\geqslant0.
\]
Hence
\[
|M+k+1+r|\geqslant |M+1+r|,
\qquad r=0,\,1,\,\ldots,\,j-1.
\]
Multiplying these inequalities and using the definition of the
Pochhammer symbol, we obtain
\[
|(M+k+1)_j|
=
\prod_{r=0}^{j-1}|M+k+1+r|
\geqslant
\prod_{r=0}^{j-1}|M+1+r|
=
|(M+1)_j|.
\]
Moreover,
\begin{align*}
&
\left|
\frac{(M)_j(M+1)_j^2}
{(2M+2)_j(M+k+1)_j\,j!}
\right|
\leqslant
\left|
\frac{(M)_j(M+1)_j}
{(2M+2)_j\,j!}
\right|\\
&=\left|\frac{\Gamma(2M+2)}{\Gamma(M)\Gamma(M+1)}\frac{\Gamma(M+j)\Gamma(M+j+1)}{\Gamma(2M+2+j)\Gamma(j+1)}\right|=O\left(\frac{1}{j^2}\right),\,j\to\infty,
\end{align*}
where the last estimate was obtained by Lemma \ref{gamma_ratio_asymptotic}. Thus the series defined by $F_k$ is absolutely convergent, and its terms are dominated by a summable sequence independent of $k$. For every fixed $j\in\N$, by Lemma \ref{gamma_ratio_asymptotic},
\begin{align*}
|(M+k+1)_j|\sim k^j,
\qquad k\to\infty,
\end{align*}
and
\begin{align*}
\left|
\frac{(M)_j(M+1)_j^2}
{(2M+2)_j(M+k+1)_j\,j!}
\right|
\to0, \qquad k\to\infty.
\end{align*}
Therefore, by dominated convergence for series, $F_k\to1$, $k\to\infty$. By Lemma \ref{gamma_ratio_asymptotic}
with $\alpha=1$ and $\beta=M+1$, it follows that
\[
\frac{k!}{(M)_{k+1}}
\sim
\frac{\Gamma(M)}{k^{M}},\qquad k\to\infty.
\]
Thus
\[
y_k=\frac{k!}{(M)_{k+1}}F_k\sim\frac{\Gamma(M)}{k^{M}},
\qquad k\to\infty,
\]
and the statement follows.
\end{proof}

\begin{lem}\label{rec_solutions_growth}
Let $\Re M>0$ and recall that the sequence $y_k,\,k\in\mathbb{N}_0$ is defined in \eqref{def_y_k}. The recurrence
\begin{equation}\label{recurrence_eq}
k^2(y_{k+1}-2y_k+y_{k-1})=M(M+1)y_k,\qquad k\in\mathbb{N},
\end{equation} has a minimal solution satisfying
\[
y_k\sim Ck^{-M},
\qquad k\to\infty,
\]
for some constant $C\neq0$.
\end{lem}

\begin{proof}
Set $k=n+1$ and rewrite \eqref{recurrence_eq} as
\begin{align}\label{rewritten}
y_{n+2}+a(n)y_{n+1}+b(n)y_n=0,\qquad n\in\mathbb{N}_0
\end{align}
where
\[
a(n)=-2-\frac{M(M+1)}{(n+1)^2},
\qquad
b(n)=1.
\]
In the notation of \cite{Wong},
\begin{align*}
a_0=-2,\qquad a_1=0,\qquad a_2=-M(M+1),\qquad
b_0=1,\qquad b_1=b_2=0,
\end{align*}
and the associated characteristic equation of \eqref{rewritten} is
\[
\rho^2+a_0\rho+b_0
=
\rho^2-2\rho+1=0,
\]
which has the double root $\rho=1$. Moreover, this root
satisfies the so-called auxiliary equation \cite[Eq. (1.6)]{Wong}
\[
a_1\rho+b_1=0.
\]
Hence, the recurrence \eqref{rewritten} falls under the
exceptional case considered in \cite[Sec.~7]{Wong}.

The corresponding indicial polynomial of \eqref{rewritten} (see \cite[Eq.~(1.8)]{Wong}) is
\[
q(\alpha)
=
\alpha(\alpha-1)\rho^2
+(a_1\alpha+a_2)\rho+b_2
=
\alpha(\alpha-1)-M(M+1),
\]
whose roots are
\[
\alpha_1=-M,
\qquad
\alpha_2=M+1.
\]
Since $\Re M>0$, we have
\[
\Re(\alpha_2-\alpha_1)=2\Re M+1>0.
\]

If $2M+1\notin\mathbb{N}$, then by the results 
of \cite[Secs.~7--8]{Wong}, there exist two linearly independent asymptotic
solutions of \eqref{rewritten} satisfying
\[
y_n^{(1)}\sim C_1n^{-M},
\qquad
y_n^{(2)}\sim C_2n^{M+1},\qquad n\to\infty,
\]
where $C_1,\,C_2\neq0$.

If $2M+1\in\mathbb{N}$, then \cite[Secs. 7 and 9]{Wong} state that one solution of \eqref{rewritten} again satisfies
\[
y_n^{(1)}\sim C_1n^{-M},\qquad C_1\neq0,\qquad n\to\infty
\]
whereas the second linearly independent solution is of the form
\[
y_n^{(2)}
=
z_n+c\,y_n^{(1)}\,\log n,
\]
where
\[
z_n\sim C_2n^{M+1},\qquad n\to\infty,
\]
with $C_2\neq0$, and $c$ is a constant that may be zero. Since $\Re M>0$ and
\[
n^{-M}\log n=o(n^{M+1}),
\qquad n\to\infty,
\]
it follows
\[
y_n^{(2)}\sim C_2n^{M+1},\qquad n\to\infty.
\]
in this case also.

Thus, in either case,
\[
\frac{y_n^{(1)}}{y_n^{(2)}}
\sim
\frac{C_1}{C_2}n^{-2M-1}
\to 0,
\qquad n\to\infty,\qquad \Re M>0.
\]
Therefore, $y_n^{(1)}$ is a minimal solution. Finally, since
$k=n+1$, $\Re M>0$ and
\[
(n+1)^{-M}\sim n^{-M},
\qquad n\to\infty,
\]
the recurrence \eqref{recurrence_eq} has a minimal solution satisfying
\[
y_k\sim Ck^{-M},
\qquad k\to\infty,
\]
for some $C\neq0$.
\end{proof}

\begin{lem}\label{hyp_to_sum}
Let $\Re M>0$. Then
\[
\frac{1}{(M+1)(M+2)}
{}_3F_2\left(
\begin{matrix}
1,\,2,\,M+1\\
M+2,\,M+3
\end{matrix}
;1
\right)
=
1-M\sum_{k=0}^{\infty}
\frac{k!}{(M+k+1)(M+2)_k}.
\]
\end{lem}

\begin{proof}
Notice that
\[
(2)_k=(k+1)!,
\qquad
\frac{(M+1)_k}{(M+2)_k}
=\frac{M+1}{M+k+1},
\qquad
(M+3)_k=\frac{M+k+2}{M+2}(M+2)_k.
\]
Hence
\begin{align*}
&\frac{1}{(M+1)(M+2)}
{}_3F_2\left(
\begin{matrix}
1,\,2,\,M+1\\
M+2,\,M+3
\end{matrix}
;1
\right)
=
\frac{1}{(M+1)(M+2)}
\sum_{k=0}^{\infty}
\frac{(2)_k(M+1)_k}
     {(M+2)_k(M+3)_k} \\
&=
\sum_{k=0}^{\infty}
\frac{k+1}
     {(M+k+1)(M+k+2)}
\frac{k!}{(M+2)_k} \\
&=
\sum_{k=0}^{\infty}
\left(
\frac{M+1}{M+k+2}
-\frac{M}{M+k+1}
\right)
\frac{k!}{(M+2)_k} \\
&=
(M+1)\sum_{k=0}^{\infty}\frac{k!}{(M+k+2)(M+2)_k}
-
M\sum_{k=0}^{\infty}
\frac{k!}{(M+k+1)(M+2)_k}=:S_1+S_2.
\end{align*}

We evaluate $S_1$ by using the beta function. Since
\begin{align*}
B(k+1,M+2)
=
\frac{k!\,\Gamma(M+2)}
     {\Gamma(M+k+3)}
=
\frac{k!}{(M+2)_{k+1}},
\end{align*}
and
\begin{align*}
\frac{1}{M+k+2}\frac{k!}{(M+2)_k}
=
\frac{k!}{(M+2)_{k+1}}
\end{align*}
we have
\begin{align*}
S_1=
\sum_{k=0}^{\infty}\frac{k!}{(M+2)_{k+1}}
&=
\sum_{k=0}^{\infty}B(k+1,M+2)
=
\sum_{k=0}^{\infty}
\int_0^1 t^k(1-t)^{M+1}\,dt\\
&=
\int_0^1(1-t)^{M+1}
\sum_{k=0}^{\infty}t^k\,dt=
\int_0^1(1-t)^{M}\,dt
=
\frac{1}{M+1}.
\end{align*}
Therefore
\[
\frac{1}{(M+1)(M+2)}
{}_3F_2\left(
\begin{matrix}
1,\,2,\,M+1\\
M+2,\,M+3
\end{matrix}
;1
\right)
=
1-
S_2,
\]
as required.
\end{proof}

\begin{lem}\label{hyp_to_lerch}
Let $\Re M>0$. Then
\[
\frac{1}{M+1}
\sum_{k=0}^{\infty}
\frac{k!}{(M+k+1)(M+2)_k}
=
2\Phi(-1,2,M+1),
\]
where $\Phi$ denotes the Lerch transcendent,
\[
\Phi(-1,2,M+1)
=
\sum_{n=0}^{\infty}
\frac{(-1)^n}{(n+M+1)^2}.
\]
\end{lem}

\begin{proof}
By the definition of the Pochhammer symbol and properties of the
beta function,
\begin{align*}
&\frac{1}{M+1}
\sum_{k=0}^{\infty}
\frac{k!}{(M+k+1)(M+2)_k}
=
\sum_{k=0}^{\infty}
\frac{\Gamma(k+1)\Gamma(M+1)}
     {(M+k+1)\Gamma(M+k+2)}\\
&=
\sum_{k=0}^{\infty}
\frac{B(M+1,k+1)}{M+k+1}
=
\sum_{k=0}^{\infty}
\frac{1}{M+k+1}
\int_0^1 y^M(1-y)^k\,dy=:SI.
\end{align*}
The identity
\[
\frac{1}{M+k+1}=\int_0^1x^{M+k}\,dx
\]
gives
\[
SI
=
\sum_{k=0}^{\infty}
\int_0^1\int_0^1
x^{M+k}y^M(1-y)^k\,dx\,dy.
\]
Since $\Re M>0$,
\begin{align*}
&\sum_{k=0}^{\infty}
\int_0^1\int_0^1
\left|x^{M+k}y^M(1-y)^k\right|\,dx\,dy\\
&=
\int_0^1\int_0^1
x^{\Re M}y^{\Re M}
\sum_{k=0}^{\infty}[x(1-y)]^k\,dx\,dy\\
&=
\int_0^1\int_0^1
\frac{x^{\Re M}y^{\Re M}}
{1-x+xy}\,dx\,dy
<\infty.
\end{align*}
Hence, by Fubini's theorem \cite[Thm.~2.37]{Folland},
\[
SI
=
\int_0^1\int_0^1
x^My^M
\sum_{k=0}^{\infty}(x(1-y))^k\,dx\,dy
=
\int_0^1\int_0^1
\frac{x^My^M}{1-x+xy}\,dx\,dy=:II.
\]
Under the change of variables $u=xy$ and $v=x$,
we have
\[
x=v,\qquad y=\frac{u}{v},
\qquad
\left|\frac{\partial(x,y)}{\partial(u,v)}\right|=\frac1v,
\]
and the region of integration becomes $0<u<v<1$. Therefore
\begin{align*}
II
&
=
\int_0^1u^M
\left(
\int_u^1
\frac{dv}{v(1+u-v)}
\right)du\\
&=
\int_0^1\frac{u^M}{1+u}
\left(
\int_u^1
\left(
\frac1v+\frac1{1+u-v}
\right)dv
\right)du=
-2\int_0^1\frac{u^M\log u}{1+u}\,du.
\end{align*}
Since
\begin{align*}
\frac1{1+u}=\sum_{n=0}^{\infty}(-u)^n,
\qquad 0\leq u<1,
\end{align*}
and
\begin{align*}
&\sum_{n=0}^{\infty}
\int_0^1
\left|u^{n+M}\log u\right|\,du
=
-\sum_{n=0}^{\infty}\int_{0}^{1}u^{n+\Re M}\log u\,du=-\sum_{n=0}^{\infty}\frac{\partial}{\partial \Re M}\int_{0}^{1}u^{n+\Re M}\,du\\
&=
-\sum_{n=0}^{\infty}\frac{\partial}{\partial \Re M}\frac{1}{n+\Re M+1}
=
\sum_{n=0}^{\infty}
\frac{1}{(n+\Re M+1)^2}
<\infty,\qquad \Re M>0,
\end{align*}
the absolute convergence estimate above allows us to integrate $II$ termwise.
Thus
\begin{align*}
II
=
-2\sum_{n=0}^{\infty}
(-1)^n
\int_0^1u^{n+M}\log u\,du
=
2\sum_{n=0}^{\infty}
\frac{(-1)^n}{(n+M+1)^2}
=
2\Phi(-1,2,M+1).
\end{align*}
\end{proof}

\begin{lem}\label{lem:relation}
Let
\[
F_n={}_3F_2\left(
\begin{matrix}
M,M+1,M+1\\
2M+2,M+n
\end{matrix};1
\right),
\qquad n=1,\,2,\,3.
\]
Then
\[
2(F_3-F_2)
+(M+1)(M+2)(F_1-F_2)
-M(M+1)^2F_2=0.
\]
\end{lem}

\begin{proof}
The identity follows from the contiguous relations for ${}_3F_2(1)$
given in \cite[p.~63]{Bailey1954}, see also \cite[Ch.~IV]{Wilson}. We verify it directly. Define
\begin{align*}
t_k=
\frac{(M)_k(M+1)_k^2}
     {(2M+2)_k(M+2)_k\,k!},
\qquad k\in\N_0.
\end{align*}
Thus
\begin{align*}
F_2=\sum_{k=0}^{\infty}t_k.
\end{align*}
Since
\begin{align*}
\frac{(M+2)_k}{(M+3)_k}
=\frac{M+2}{M+k+2},
\end{align*}
we have
\begin{align*}
F_3
=
\sum_{k=0}^{\infty}
\frac{M+2}{M+k+2}\,t_k.
\end{align*}
Likewise,
\begin{align*}
\frac{(M+2)_k}{(M+1)_k}
=\frac{M+k+1}{M+1},
\end{align*}
and hence
\begin{align*}
F_1
=
\sum_{k=0}^{\infty}
\frac{M+k+1}{M+1}\,t_k.
\end{align*}
Consequently,
\begin{align*}
&2(F_3-F_2)
 +(M+1)(M+2)(F_1-F_2)
 -M(M+1)^2F_2
=
\sum_{k=0}^{\infty} A_k t_k,
\end{align*}
where
\[
A_k=
2\left(\frac{M+2}{M+k+2}-1\right)
+(M+1)(M+2)
\left(\frac{M+k+1}{M+1}-1\right)
-M(M+1)^2.
\]
After simplification,
\[
A_k
=
\frac{(M+2)
\left(
-M^3-M^2k-2M^2+Mk-M+k^2+k
\right)}
{M+k+2}.
\]
On the other hand, from the definition of \(t_k\),
\[
\frac{t_{k+1}}{t_k}
=
\frac{(M+k)(M+k+1)^2}
     {(2M+k+2)(M+k+2)(k+1)}.
\]
It follows by direct simplification that
\[
A_k t_k
=
-(k+1)(M+2)(2M+k+2)t_{k+1}
+k(M+2)(2M+k+1)t_k.
\]
Define
\[
G_k=-k(M+2)(2M+k+1)t_k.
\]
Then
\[
A_k t_k=G_{k+1}-G_k,
\]
and therefore
\begin{align}\label{toshow0}
2(F_3-F_2)
 +(M+1)(M+2)(F_1-F_2)
 -M(M+1)^2F_2
=
\sum_{k=0}^{\infty}(G_{k+1}-G_k).
\end{align}
It remains to show that the right-hand side of \eqref{toshow0} equals zero.  Since
\begin{align*}
t_k=
\frac{\Gamma(2M+2)\Gamma(M+2)}
{\Gamma(M)\Gamma(M+1)^2}
\frac{\Gamma(M+k)}{\Gamma(k+1)}
\frac{\Gamma(M+k+1)}{\Gamma(M+k+2)}
\frac{\Gamma(M+k+1)}{\Gamma(2M+k+2)},
\end{align*}
by Lemma \ref{gamma_ratio_asymptotic}
\begin{align*}
\frac{\Gamma(M+k)}{\Gamma(k+1)}\sim k^{M-1},
\qquad
\frac{\Gamma(M+k+1)}{\Gamma(M+k+2)}\sim\frac{1}{k},
\qquad
\frac{\Gamma(M+k+1)}{\Gamma(2M+k+2)}\sim\frac{1}{k^{M+1}},\qquad k\to\infty,
\end{align*}
and consequently
\begin{align*}
t_k=O(k^{-3}),\qquad k\to\infty.
\end{align*}
Therefore,
\begin{align*}
G_k
=
-k(M+2)(2M+k+1)t_k
=O(k^{-1})\to0,\qquad k\to\infty.
\end{align*}
Since $G_0=0$, we obtain
\begin{align*}
\sum_{k=0}^{\infty}(G_{k+1}-G_k)
=
\lim_{N\to\infty}(G_{N+1}-G_0)
=0.
\end{align*}
and the proof follows.
\end{proof}

\subsection{Proof of Theorem \ref{conj_2}}
Let $\Re M>0$. According to the conditions of Theorem
\ref{conj_2},
\[
b_0=0,\qquad
b_{2n-1}=M+1,\qquad
b_{2n}=\frac{M}{n^2},
\qquad n\in\mathbb{N}.
\]
Since $\Re M>0$, both $M$ and $M+1$ lie in the open right
half-plane. Hence there exists $\varepsilon\in(0,\pi/2)$ such that
\[
|\arg M|\leqslant\frac{\pi}{2}-\varepsilon,
\qquad
|\arg(M+1)|\leqslant\frac{\pi}{2}-\varepsilon.
\]
Consequently, 
\begin{align*}
b_n\in G_\varepsilon=\left\{w\in\mathbb{C}:|\arg w|\leqslant
\frac{\pi}{2}-\varepsilon\right\}\cup\{0\}
,\qquad n\in\mathbb{N},
\end{align*} and $b_1=M+1\neq0$. Moreover,
\[
\sum_{n=1}^{\infty}|b_n|
\geqslant
\sum_{n=1}^{\infty}|b_{2n-1}|
=
\sum_{n=1}^{\infty}|M+1|
=\infty.
\]
Therefore, by Lemma \ref{cf_convergence_1/bn}, the continued fraction
$b_0+\k(1/b_n)$ converges. Denote its limit by $L$. The odd part of the continued fraction $b_0+\k(1/b_n)$ (see Lemma \ref{odd_part}) is
\begin{align}\label{odd_part_expr}
\frac{1}{M+1}-\frac{1}{(M+1)(2+M(M+1)+K)},
\end{align}
where
\begin{align}\label{tail_cf}
K=\K_{n=2}^{\infty}\frac{-1}{2+\frac{M(M+1)}{n^2}}=\K_{n=1}^{\infty}\frac{-1}{2+\frac{M(M+1)}{(n+1)^2}}.
\end{align}
Because $\k(1/b_n)$ converges, the convergents of its odd part have the same limit $L$. Consider the tail $K$. The associated recurrence relation of \eqref{tail_cf} is:
\begin{align*}
x_n-\left(2+\frac{M(M+1)}{(n+1)^2}\right)x_{n-1}+x_{n-2}=0,\qquad n\in\N.
\end{align*}
Let $k=n+1$ and define $y_k=x_{k-2}$. The sequence $y_k$ satisfies the recurrence relation
\begin{align}\label{cf_recurrence}
k^2(y_{k+1}-2y_k+y_{k-1})=M(M+1)y_k,
\qquad k\in\mathbb{N}.
\end{align}
Due to Lemma \ref{recurrence_sol_lem}, the recurrence \eqref{cf_recurrence} has a solution
\begin{align*}
y_k=\frac{k!}{(M)_{k+1}} \Fhyp{M,\,M+1,\,M+1}{2M+2,\,M+k+1}{1},\qquad k\in\mathbb{N}_0.
\end{align*}
By Lemmas \ref{solution_asymptotics} and \ref{rec_solutions_growth}, this solution is minimal and $y_0\neq0$. Thus, applying Pincherle Theorem, Lemma \ref{Pincherle}, we obtain
\begin{align}\label{tail_value_hyp}
K=-\frac{x_0}{x_{-1}}=-\frac{y_2}{y_1}=-\frac{2}{M+2}\frac{{}_3F_2(M,\,M+1,\,M+1;\,2M+2,\,M+3;\,1)}{{}_3F_2(M,\,M+1,\,M+1;\,2M+2,\,M+2;\,1)}.
\end{align}
Inserting \eqref{tail_value_hyp} into \eqref{odd_part_expr}, we obtain
\begin{align}\label{cf_value_raw}
L=\frac{1}{M+1}-\frac{1}{(M+1)\left(2+M(M+1)-\frac{2}{M+2}\frac{F_3}{F_2}\right)},
\end{align}
where $F_n={}_3F_2(M,\,M+1,\,M+1;\,2M+2,\,M+n;\,1)$, $n=1,\,2,\,3$. The functions $F_1,\,F_2$ and $F_3$ satisfy (see Lemma \ref{lem:relation})
\begin{align*}
2(F_3-F_2)+(M+1)(M+2)(F_1-F_2)-M(M+1)^2F_2=0,
\end{align*}
or, equivalently,
\begin{align*}
2\frac{F_3}{F_2}=2+(M+1)(M+2)+M(M+1)^2-(M+1)(M+2)\frac{F_1}{F_2}.
\end{align*}
From this, part of the denominator in \eqref{cf_value_raw} simplifies to
\begin{align*}
2+M(M+1)-\frac{2}{M+2}\frac{F_3}{F_2}=(M+1)\frac{F_1}{F_2}.
\end{align*}
Using theorem due to Gauss (see \cite[Thm. 1.3]{Bailey}) 
\begin{align*}
F_1&=\Fhyp{M,\,M+1,\,M+1}{2M+2,\,M+1}{1}={}_2F_1(M,\,M+1;\,2M+2;\,1)=\frac{\Gamma(2M+2)}{\Gamma(M+2)\Gamma(M+1)}\\
&=\frac{\Gamma(2M+2)}{(M+1)(\Gamma(M+1))^2}
\end{align*}
we obtain
\begin{align*}
L=\frac{1}{M+1}\left(1-\frac{(\Gamma(M+1))^2\,F_2}{\Gamma(2M+2)}\right).
\end{align*}
Applying the Thomae transformation \eqref{thomae}, yields
\begin{align*}
F_2&={}_3F_2(M+1,\,M+1,\,M;\,2M+2,\,M+2;\,1) \\
&=\frac{\Gamma(2M+2)}{\Gamma(M+1)\Gamma(M+3)} \Fhyp{M+1,\,1,\,2}{M+3,\,M+2}{1}
\end{align*}
and
\begin{align*}\label{cf_value_simplified}
L=\frac{1}{M+1}\left(1-\frac{{}_3F_2(1,\,2,\,M+1;\,M+2,\,M+3;\,1)}{(M+1)(M+2)}\right).
\end{align*}
By Lemma \ref{hyp_to_sum}, we have
\begin{align*}
L=\frac{M}{M+1}\sum_{k=0}^{\infty}\frac{k!}{(M+k+1)(M+2)_k},
\end{align*}
and due to Lemma \ref{hyp_to_lerch} the proof follows.
\qed

\section{Proofs of the third and fourth theorems}

We prove Theorems \ref{conj_3} and \ref{conj_4} in one section. As before, we start with the auxiliary statements.

\subsection{Auxiliary statements to prove Theorems \ref{conj_3} and \ref{conj_4}}

\begin{lem}\label{polygamma1}
Let $\psi_1$ denote the polygamma function of order $1$, also known as the trigamma function,
\[
\psi_1(z)=\sum_{n=0}^{\infty}\frac{1}{(n+z)^2}.
\]
Then
\[
\psi_1(z)
=
\frac{1}{z}
+\frac{1}{2z^2}
+\frac{2\pi}{z}g(z),
\]
where
\[
g(z)=
\cfrac{a_1}{z^2+
\cfrac{a_2}{1+
\cfrac{a_3}{z^2+
\cfrac{a_4}{1+\cdots}}}},
\qquad
|\arg z|<\frac{\pi}{2},
\]
and
\begin{align*}
a_1=\frac{1}{12\pi},
\qquad
a_n=\frac{n^2(n^2-1)}{4(4n^2-1)},
\qquad n\in \{2,\,3,\,\ldots\}.
\end{align*}
\end{lem}

\begin{proof}
See \cite[p.~235]{Handbook} for the stated formulation and
\cite[pp.~240-245]{Lange1994} for its derivation and proof.
\end{proof}

\begin{lem}\label{polygamma2}
Let $\psi_2$ denote the polygamma function of order 2, also known as the tetragamma function,
\[
\psi_2(z)
=
-2\sum_{n=0}^{\infty}\frac{1}{(n+z)^3}.
\]
Then
\[
\psi_2(z)
=
-\frac{1}{z^2}
-\frac{1}{z^3}
-\left(\frac{2\pi}{z}\right)^2 g(z),
\]
where
\[
g(z)
=
\cfrac{a_1}{z^2+
\cfrac{a_2}{1+
\cfrac{a_3}{z^2+
\cfrac{a_4}{1+\cdots}}}},
\qquad
|\arg z|<\frac{\pi}{2},
\]
and
\begin{align*}
a_1=\frac{1}{8\pi^2},
\qquad
a_{2n}
=
\frac{n^2(n+1)}{2(2n+1)},
\qquad
a_{2n+1}
=
\frac{n(n+1)^2}{2(2n+1)},
\qquad n\in\mathbb{N}.
\end{align*}
\end{lem}

\begin{proof}
See \cite[p.~235]{Handbook} for the stated formulation and
\cite[pp.~245--249]{Lange1994} for its derivation and proof.
\end{proof}

\begin{lem}\label{polygamma_recurrence}
Let $n\in\mathbb{N}$ and let $\psi_n$ denote the polygamma
function of order $n$
\[
\psi_n(z)
=
(-1)^{n+1}n!
\sum_{k=0}^{\infty}\frac{1}{(k+z)^{n+1}}.
\]
Then, for
$z\notin\mathbb{Z}_0^-=\{0,-1,-2,\ldots\}$,
\[
\psi_n(z+1)
=
\psi_n(z)
+
\frac{(-1)^n n!}{z^{n+1}},\qquad n\in\mathbb{N}.
\]
\end{lem}

\begin{proof}
The proof is straightforward
\[
\psi_n(z+1)
=
(-1)^{n+1}n!
\sum_{k=1}^{\infty}\frac{1}{(k+z)^{n+1}}
=
\psi_n(z)+\frac{(-1)^n n!}{z^{n+1}}.
\]
\end{proof}
\subsection{Proof of Theorem \ref{conj_3}}
By the conditions of Theorem \ref{conj_3},
\begin{align*}
b_0=2M-1,\qquad
b_n=2M\frac{2n+1}{n(n+1)},\qquad n\in\mathbb N.
\end{align*}
Since $\Re M>0$ and $(2n+1)/(n(n+1))>0$, by Lemma~\ref{cf_convergence_1/bn}
there exists $\varepsilon\in(0,\pi/2)$ such that
\begin{align*}
b_n\in G_\varepsilon=\left\{w\in\mathbb{C}:|\arg w|\leqslant
\frac{\pi}{2}-\varepsilon\right\}\cup\{0\}, \qquad n\in\N,
\end{align*}
Moreover,
\[
|b_n|
=
2|M|\frac{2n+1}{n(n+1)}
\sim\frac{4|M|}{n},
\qquad n\to\infty,
\]
implies
\[
\sum_{n=1}^{\infty}|b_n|=\infty.
\]
Therefore, by Lemma~\ref{cf_convergence_1/bn}, the continued fraction $b_0+\k(1/b_n)$
converges. Denote its limit by $L$. Then
\begin{align}\label{to_substitute}
L=2M-1+\K_{n=1}^{\infty}\frac{1}{b_n}.
\end{align}
Applying the equivalence transformation from Lemma~\ref{equiv_cf}
with $s_0=1$ and
\[
s_n=
\begin{cases}
\displaystyle
M\frac{n(n+1)}{4n+2}, & n\ \text{odd},\\[3mm]
\displaystyle
\frac{1}{M}\frac{n(n+1)}{4n+2}, & n\ \text{even},
\end{cases}
\qquad n\in\N,
\]
we obtain
\[
K:=\K_{n=1}^{\infty}\frac{1}{b_n}
=
\K_{n=1}^{\infty}\frac{e_n}{f_n},
\]
where
\[
e_n=s_ns_{n-1},\qquad f_n=s_nb_n, \qquad s_0=1, \qquad n\in\mathbb{N}.
\]
Consequently,
\[
e_1=\frac{M}{3},
\qquad
e_n=\frac{n^2(n^2-1)}{4(4n^2-1)},
\qquad n\in \{2,\,3,\,\ldots\},
\]
and
\[
f_n=
\begin{cases}
M^2,&n\ \text{odd},\\
1,&n\ \text{even},
\end{cases}
\qquad n\in\mathbb{N}.
\]
By Lemma~\ref{polygamma1},
\[
\psi_1(M)
=
\frac{1}{M}
+\frac{1}{2M^2}
+\frac{2\pi}{M}g(M),\qquad \Re M>0,
\]
where
\[
g(M)=
\cfrac{a_1}{
M^2+
\cfrac{a_2}{
1+
\cfrac{a_3}{
M^2+
\cfrac{a_4}{1+\ddots}}}},
\]
with
\[
a_1=\frac{1}{12\pi},
\qquad
a_n=\frac{n^2(n^2-1)}{4(4n^2-1)},
\qquad n\in \{2,\,3,\,\ldots\}.
\]
On the other hand, the continued fraction defining $K$ can be written as
\[
K=
\cfrac{M/3}{
M^2+
\cfrac{a_2}{
1+
\cfrac{a_3}{
M^2+
\cfrac{a_4}{1+\ddots}}}}.
\]
Therefore,
\[
g(M)
=
\frac{1/(12\pi)}{M/3}\,K
=
\frac{K}{4\pi M}.
\]
Consequently,
\begin{align*}
\psi_1(M)
=
\frac{1}{M}
+\frac{1}{2M^2}
+\frac{K}{2M^2}
\qquad \Rightarrow \qquad
K
=
2M^2
\left(
\psi_1(M)-\frac{1}{M}-\frac{1}{2M^2}
\right).
\end{align*}
By Lemma~\ref{polygamma_recurrence},
\[
\psi_1(M)=\psi_1(M+1)+\frac{1}{M^2}.
\]
Therefore
\begin{align*}
K
=
2M^2\psi_1(M+1)-2M+1.
\end{align*}
Since
\[
\psi_1(M+1)
=
\sum_{n=0}^{\infty}\frac{1}{(n+M+1)^2}
=
\Phi(1,2,M+1),
\]
we obtain
\[
K=2M^2\Phi(1,2,M+1)-2M+1.
\]
Substituting this into \eqref{to_substitute} yields
\[
L
=
2M^2\Phi(1,2,M+1),
\]
which proves the theorem.
\qed

\subsection{Proof of Theorem \ref{conj_4}}

By the conditions of Theorem~\ref{conj_4},
\begin{align*}
b_0&=M-1,\qquad
b_{2n-1}=\frac{2M}{n},\qquad
b_{2n}=\frac{M(2n+1)}{n(n+1)},
\qquad n\in\mathbb N.
\end{align*}
Equivalently,
\[
b_n=
\begin{cases}
\dfrac{4M}{n+1}, & n\ \text{odd},\\[2mm]
\dfrac{4M(n+1)}{n(n+2)}, & n\ \text{even}.
\end{cases}
\]
As in the proof of Theorem~\ref{conj_3}, since $\Re M>0$, there exists
$\varepsilon\in(0,\pi/2)$ such that
\[
|\arg M|\leqslant\frac{\pi}{2}-\varepsilon.
\]
Since each $b_n$ is a positive real multiple of $M$, we have
\[
b_n\in
G_\varepsilon
:=
\left\{
z\in\mathbb C:
|\arg z|\leqslant\frac{\pi}{2}-\varepsilon
\right\}\cup\{0\},
\qquad n\in\N.
\]
Moreover,
\[
|b_n|\sim\frac{4|M|}{n},
\qquad n\to\infty,
\]
implies
\[
\sum_{n=1}^{\infty}|b_n|=\infty.
\]
Therefore, by Lemma~\ref{cf_convergence_1/bn}, the continued fraction $b_0+\k(1/b_n)$
converges. Denote its limit by $L$. Thus,
\[
L=M-1+K, \qquad K=\K_{n=1}^{\infty}\frac{1}{b_n}.
\]

Applying the equivalence transformation from Lemma~\ref{equiv_cf}
with $s_0=1$ and
\[
s_n=
\begin{cases}
\displaystyle
M\frac{n+1}{4}, & n\ \text{odd},\\[3mm]
\displaystyle
\frac{n(n+2)}{4M(n+1)}, & n\ \text{even},
\end{cases}
\qquad n\in\N,
\]
we obtain
\[
K=\K_{n=1}^{\infty}\frac{1}{b_n}
=
\K_{n=1}^{\infty}\frac{e_n}{f_n}
,
\]
where
\[
e_n=s_ns_{n-1},\qquad
f_n=s_nb_n, \qquad s_0=1, \qquad n\in\mathbb{N}.
\]
Consequently,
\begin{align*}
e_1=\frac{M}{2},
\qquad
e_{2n}
=
\frac{n^2(n+1)}{2(2n+1)},
\qquad
e_{2n+1}
=
\frac{n(n+1)^2}{2(2n+1)},\qquad n\in\N
\end{align*}
and
\[
f_n=
\begin{cases}
M^2,&n\ \text{odd},\\
1,&n\ \text{even},
\end{cases}
\qquad n\in\mathbb{N}.
\]
By Lemma~\ref{polygamma2},
\[
\psi_2(M)
=
-\frac{1}{M^2}
-\frac{1}{M^3}
-\left(\frac{2\pi}{M}\right)^2g(M), \qquad \Re M>0,
\]
where
\[
g(M)
=
\cfrac{a_1}{
M^2+\cfrac{a_2}{
1+\cfrac{a_3}{
M^2+\cfrac{a_4}{1+\ddots}}}},
\]
with
\[
a_1=\frac{1}{8\pi^2},
\qquad
a_{2n}=e_{2n},
\qquad
a_{2n+1}=e_{2n+1}, \qquad n\in\mathbb{N}.
\]
Comparing the continued fractions defining $g(M)$ and $K$, we obtain
\[
g(M)
=
\frac{1/(8\pi^2)}{M/2}\,K
=
\frac{K}{4\pi^2M}.
\]
Consequently,
\begin{align*}
\psi_2(M)
=
-\frac{1}{M^2}
-\frac{1}{M^3}
-\frac{K}{M^3},
\qquad \Rightarrow \qquad
K
=
M^3\left(
-\psi_2(M)-\frac{1}{M^2}-\frac{1}{M^3}
\right).
\end{align*}
By Lemma~\ref{polygamma_recurrence},
\[
\psi_2(M+1)
=
\psi_2(M)+\frac{2}{M^3},
\]
and therefore
\begin{align*}
K
=
-M^3\psi_2(M+1)-M+1.
\end{align*}
By using,
\[
\psi_2(M+1)
=
-2\sum_{n=0}^{\infty}
\frac{1}{(n+M+1)^3}
=
-2\Phi(1,3,M+1),
\]
we obtain
\[
K
=
2M^3\Phi(1,3,M+1)-M+1.
\]
Therefore
\begin{align*}
L=M-1+K=
2M^3\Phi(1,3,M+1),
\end{align*}
which proves the theorem.
\qed

\bibliographystyle{plainurl}
\bibliography{bibliography}

\begin{thebibliography}{10}

\bibitem{Arakawa}
Tsuneo Arakawa, Tomoyoshi Ibukiyama, and Masanobu Kaneko.
\newblock {\em Bernoulli Numbers and Zeta Functions}.
\newblock Springer, Tokyo, 2014.
\newblock \href {https://doi.org/10.1007/978-4-431-54919-2} {\path{doi:10.1007/978-4-431-54919-2}}.

\bibitem{Bailey}
W.~N. Bailey.
\newblock {\em Generalized Hypergeometric Series}.
\newblock Cambridge Tracts in Mathematics and Mathematical Physics. The University Press, Cambridge, 1935.

\bibitem{Bailey1954}
W.~N. Bailey.
\newblock Contiguous hypergeometric functions of the type ${}_3{F}_2(1)$.
\newblock {\em Glasg. Math. J.}, 2(2):62--65, 1954.
\newblock \href {https://doi.org/10.1017/S2040618500033049} {\path{doi:10.1017/S2040618500033049}}.

\bibitem{Berg}
C.~Berg.
\newblock {\em Complex Analysis}.
\newblock Department of Mathematical Sciences, University of Copenhagen, Copenhagen, 2012.

\bibitem{BerndtRamanujanNotebooksII}
B.~C. Berndt.
\newblock {\em Ramanujan's Notebooks, Part II}.
\newblock Springer-Verlag, New York, 1989.

\bibitem{BerndtLamphereWilson1985}
B.~C. Berndt, R.~L. Lamphere, and B.~M. Wilson.
\newblock Chapter 12 of {R}amanujan's second notebook: Continued fractions.
\newblock {\em Rocky Mountain J. Math.}, 15(2):235--310, 1985.
\newblock \href {https://doi.org/10.1216/RMJ-1985-15-2-235} {\path{doi:10.1216/RMJ-1985-15-2-235}}.

\bibitem{BonanHamadaJones2005}
C.~M. Bonan-Hamada and W.~B. Jones.
\newblock Stieltjes continued fractions for polygamma functions; speed of convergence.
\newblock {\em J. Comput. Appl. Math.}, 179(1--2):47--55, 2005.
\newblock \href {https://doi.org/10.1016/j.cam.2004.09.034} {\path{doi:10.1016/j.cam.2004.09.034}}.

\bibitem{Cohen2026}
H.~Cohen.
\newblock Continued fractions of polynomial type: Theory and encyclopedic dictionary, 2026.
\newblock \href {https://arxiv.org/abs/2607.06581} {\path{arXiv:2607.06581}}.

\bibitem{Handbook}
A.~Cuyt, V.~B. Petersen, B.~Verdonk, H.~Waadeland, and W.~B. Jones.
\newblock {\em Handbook of Continued Fractions for Special Functions}.
\newblock Springer, Dordrecht, 2008.
\newblock \href {https://doi.org/10.1007/978-1-4020-6949-9} {\path{doi:10.1007/978-1-4020-6949-9}}.

\bibitem{Folland}
G.~B. Folland.
\newblock {\em Real Analysis: Modern Techniques and Their Applications}.
\newblock John Wiley \& Sons, New York, 2 edition, 1999.

\bibitem{Gautschi1967}
W.~Gautschi.
\newblock Computational aspects of three-term recurrence relations.
\newblock {\em SIAM Rev.}, 9(1):24--82, 1967.
\newblock \href {https://doi.org/10.1137/1009002} {\path{doi:10.1137/1009002}}.

\bibitem{Gutnik2010}
L.~Gutnik.
\newblock Elementary proof of {Y}. {V}. {N}esterenko expansion of the number {Z}eta(3) in continued fraction.
\newblock {\em Adv. Difference Equ.}, 2010:143521, 2010.
\newblock \href {https://doi.org/10.1155/2010/143521} {\path{doi:10.1155/2010/143521}}.

\bibitem{Lando}
S.~K. Lando.
\newblock {\em Lectures on Generating Functions}, volume~23 of {\em Student Mathematical Library}.
\newblock American Mathematical Society, Providence, RI, 2003.

\bibitem{Lange1994}
L.~J. Lange.
\newblock Continued fraction representations for functions related to the gamma function.
\newblock In S.~C. Cooper and W.~J. Thron, editors, {\em Continued Fractions and Orthogonal Functions: Theory and Applications}, volume 154 of {\em Lecture Notes in Pure and Applied Mathematics}, pages 233--279. Marcel Dekker, New York, 1994.

\bibitem{LaurincikasGarunkstis2002}
A.~Laurin{\v{c}}ikas and R.~Garunk{\v{s}}tis.
\newblock {\em The Lerch Zeta-Function}.
\newblock Kluwer Academic Publishers, Dordrecht, 1 edition, 2002.
\newblock \href {https://doi.org/10.1007/978-94-017-6401-8} {\path{doi:10.1007/978-94-017-6401-8}}.

\bibitem{Lorentzen}
L.~Lorentzen and H.~Waadeland.
\newblock {\em Continued Fractions}.
\newblock Atlantis Studies in Mathematics for Engineering and Science. Atlantis Press, Paris, 2008.

\bibitem{Matiyasevich}
Y.~V. Matiyasevich.
\newblock Calculation of {R}iemann's zeta function via interpolating determinants.
\newblock Technical Report 2013-18, Max-Planck-Institut f{\"u}r Mathematik, Bonn, 2013.
\newblock URL: \url{https://archive.mpim-bonn.mpg.de/id/eprint/2748/}.

\bibitem{DLMF}
{NIST Digital Library of Mathematical Functions}.
\newblock Nist digital library of mathematical functions.
\newblock \url{https://dlmf.nist.gov/}, 2026.
\newblock F.~W.~J. Olver, A.~B. Olde Daalhuis, D.~W. Lozier, B.~I. Schneider, R.~F. Boisvert, C.~W. Clark, B.~R. Miller, B.~V. Saunders, H.~S. Cohl, and M.~A. McClain, eds.

\bibitem{Perr57}
O.~Perron.
\newblock {\em Die Lehre von den Kettenbr{\"u}chen}, volume~2.
\newblock B. G. Teubner, Stuttgart, 3 edition, 1957.

\bibitem{Rakha}
M.~A. Rakha, A.~K. Rathie, and P.~Chopra.
\newblock On some new contiguous relations for the gauss hypergeometric function with applications.
\newblock {\em Comput. Math. Appl.}, 61(3):620--629, 2011.
\newblock \href {https://doi.org/10.1016/j.camwa.2010.12.008} {\path{doi:10.1016/j.camwa.2010.12.008}}.

\bibitem{Shenton1983}
L.~R. Shenton.
\newblock Continued fractions and the polygamma functions.
\newblock {\em J. Comput. Appl. Math.}, 9(1):29--39, 1983.
\newblock \href {https://doi.org/10.1016/0377-0427(83)90026-2} {\path{doi:10.1016/0377-0427(83)90026-2}}.

\bibitem{ShentonBowman1971}
L.~R. Shenton and K.~O. Bowman.
\newblock Continued fractions for the {PSI} function and its derivatives.
\newblock {\em SIAM J. Appl. Math.}, 20(4):547--554, 1971.
\newblock \href {https://doi.org/10.1137/0120057} {\path{doi:10.1137/0120057}}.

\bibitem{Slater}
L.~J. Slater.
\newblock {\em Generalized Hypergeometric Functions}.
\newblock Cambridge University Press, Cambridge, 1966.

\bibitem{Wilson}
J.~A. Wilson.
\newblock Three-term contiguous relations and some new orthogonal polynomials.
\newblock In E.~B. Saff and R.~S. Varga, editors, {\em Pad{\'e} and Rational Approximation}, pages 227--232. Academic Press, New York, 1977.
\newblock \href {https://doi.org/10.1016/B978-0-12-614150-4.50024-1} {\path{doi:10.1016/B978-0-12-614150-4.50024-1}}.

\bibitem{Wong}
R.~Wong and H.~Li.
\newblock Asymptotic expansions for second-order linear difference equations.
\newblock {\em J. Comput. Appl. Math.}, 41(1):65--94, 1992.
\newblock \href {https://doi.org/10.1016/0377-0427(92)90239-T} {\path{doi:10.1016/0377-0427(92)90239-T}}.

\bibitem{Zagier}
Don Zagier.
\newblock Curious and exotic identities for bernoulli numbers.
\newblock An Appendix in \cite{Arakawa}, pages 239--262.

\end{thebibliography}

\end{document}